\documentclass[11pt]{amsart}

\usepackage[letterpaper,margin=1in]{geometry}

\usepackage{amssymb,amsmath,amsfonts,graphicx,amsthm,mathrsfs,amsbsy}
\usepackage{algorithm}
\usepackage{algorithmicx}
\usepackage{algpseudocode}
\usepackage{color}

\usepackage{bmpsize}

\usepackage{amscd,verbatim}
\usepackage{enumerate}
\usepackage{float}
\usepackage[caption = false]{subfig}
\usepackage{bm}

\newtheorem{theorem}{Theorem}[section]
\newtheorem{definition}[theorem]{Definition}
\newtheorem{lemma}[theorem]{Lemma}

\newtheorem{proposition}[theorem]{Proposition}

\newtheorem{corollary}[theorem]{Corollary}

\numberwithin{equation}{section}

\def\c1{\mathbf 1}

\newcommand{\R}[1]{{\rm I\! R}^{#1}}

\DeclareGraphicsExtensions{.png,.jpg,.pdf,.eps}

\everymath{\displaystyle}

\begin{document}

\markboth{J.~C.~URSCHEL AND L.~T.~ZIKATANOV}{Energy Minimizing Spring Embeddings of Planar Graphs}

\title{Discrete Trace Theorems and Energy Minimizing Spring Embeddings of Planar Graphs}

\author{John C. Urschel}
\address{Department of Mathematics, Massachusetts Institute of Technology, Cambridge, MA, USA.}
\email{urschel@mit.edu \\ Corresponding author.}
\author{Ludmil T. Zikatanov}
\address{Department of Mathematics, The Pennsylvania State University, University Park, Pennsylvania, 16802, USA; Institute for Mathematics and Informatics, Bulgarian Academy of Sciences, Sofia, Bulgaria.}
\email{ludmil@psu.edu}

\begin{abstract}
  Tutte's spring embedding theorem states that, for a three-connected planar graph, if the outer face of the graph is fixed as the complement of some convex region in the plane, and all other vertices are placed at the mass center of their neighbors, then this results in a unique embedding, and this embedding is planar. It also follows fairly quickly that this embedding minimizes the sum of squared edge lengths, conditional on the embedding of the outer face. However, it is not at all clear how to embed this outer face. We consider the minimization problem of embedding this outer face, up to some normalization, so that the sum of squared edge lengths is minimized. In this work, we show the connection between this optimization problem and the Schur complement of the graph Laplacian with respect to the interior vertices. We prove a number of discrete trace theorems, and, using these new results, show the spectral equivalence of this Schur complement with the boundary Laplacian to the one-half power for a large class of graphs. Using this result, we give theoretical guarantees for this optimization problem, which motivates an algorithm to embed the outer face of a spring embedding.
\end{abstract}

\maketitle

\section{Introduction}

Graph drawing is an area at the intersection of mathematics, computer
science, and more qualitative fields. Despite the extensive literature
in the field, in many ways the concept of what constitutes the optimal drawing of a
graph is heuristic at best, and subjective at worst. For a general review of the major areas of research in graph drawing, we refer the reader to \cite{battista1998graph,kaufmann2003drawing}. When energy (i.e. Hall's energy, the sum of squared distances between adjacent vertices) minimization is desired, the optimal embedding in the
plane is given by the two-dimensional diffusion map induced by the
eigenvectors of the two smallest non-zero eigenvalues of the graph
Laplacian \cite{MR2154691, MR2063526, MR2029596}. This general class of
graph drawing techniques is referred to as spectral layouts. When drawing a planar graph, often a planar embedding (a drawing in which edges do not intersect) is desirable. However, spectral layouts of planar graphs are not guaranteed to be planar. When looking at
triangulations of a given domain, it is commonplace for the
near-boundary points of the spectral layout to ``grow" out of the boundary, or lack any resemblance to a planar embedding. For instance, see the spectral layout of a random triangulation of a disk and rectangle in Figure \ref{fig1}.

In his 1962 work titled ``How to Draw a Graph," Tutte found an elegant technique to produce planar embeddings of planar graphs that also minimize ``energy" in some sense \cite{MR0158387}. In particular, for a three-connected planar graph, he showed that if the outer face of the graph is fixed as the complement of some convex region in the plane, and every other point is located at the mass center of its neighbors, then the resulting embedding is planar. This embedding minimizes Hall's energy, conditional on the embedding of the boundary face. This result is now known as Tutte's spring embedding theorem, and this general
class of graph drawing techniques is known as force-based layouts. While this result is well known (see \cite{knudson_lamb}, for example), it is not so obvious how to embed the outer face. This, of course, should vary from case to case, depending on the dynamics of the interior.

\begin{figure}
	\begin{center}
		\subfloat[ Circle]{\includegraphics*[width=1.75 in,height = 1.75in]{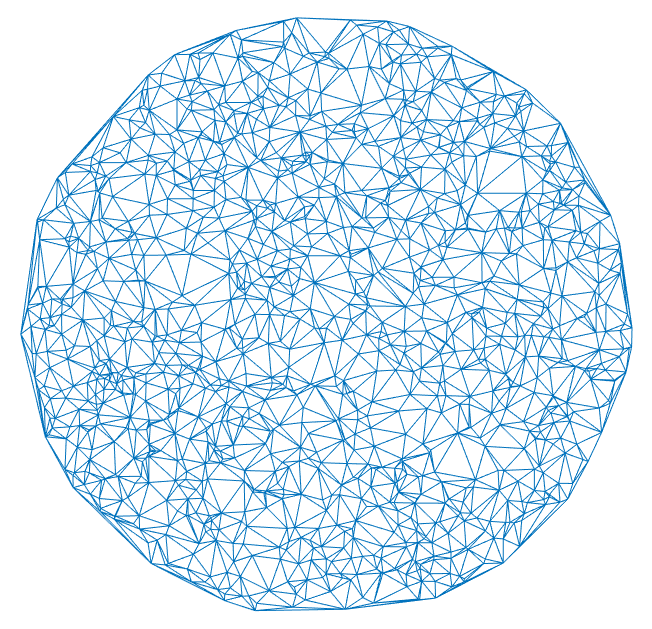}} \qquad
		\subfloat[$3$-by-$1$ Rectangle]{\includegraphics*[width= 3.5 in, height = 1.5 in]{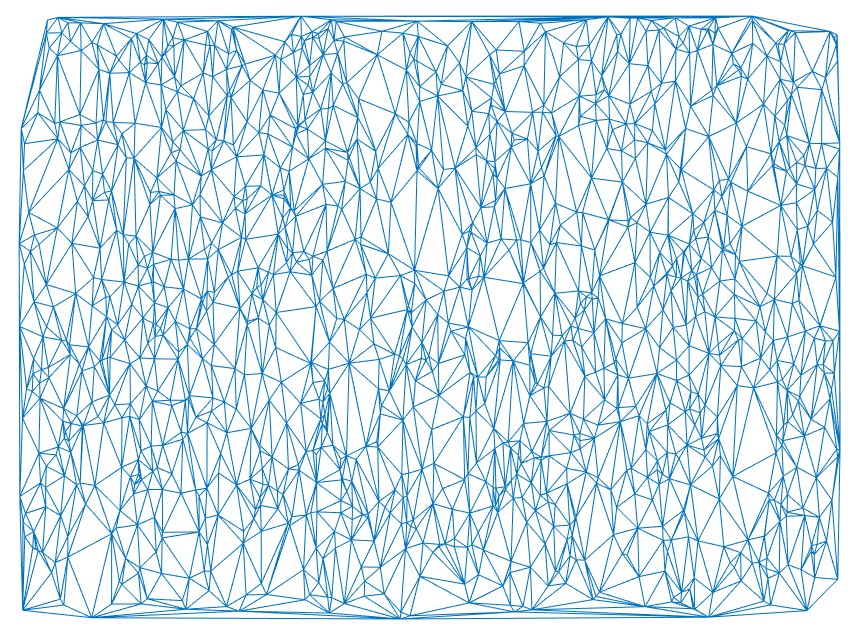}} \\
		\subfloat[Spectral Layout]{\includegraphics*[width=1.75 in,height = 1.75in]{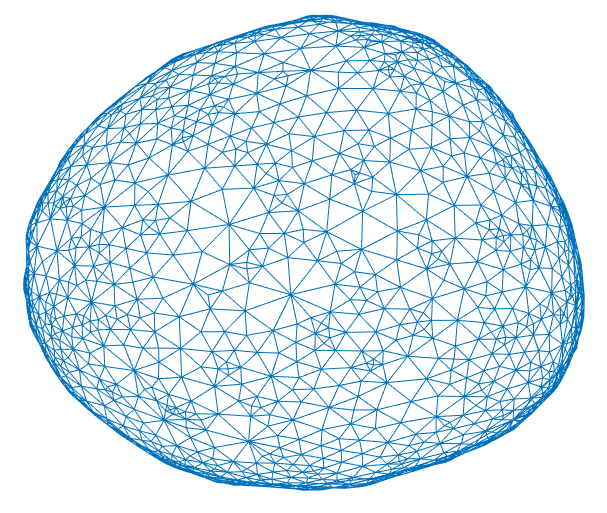}} \qquad
		\subfloat[Spectral Layout]{\includegraphics*[width= 3.5 in, height = 1.5 in]{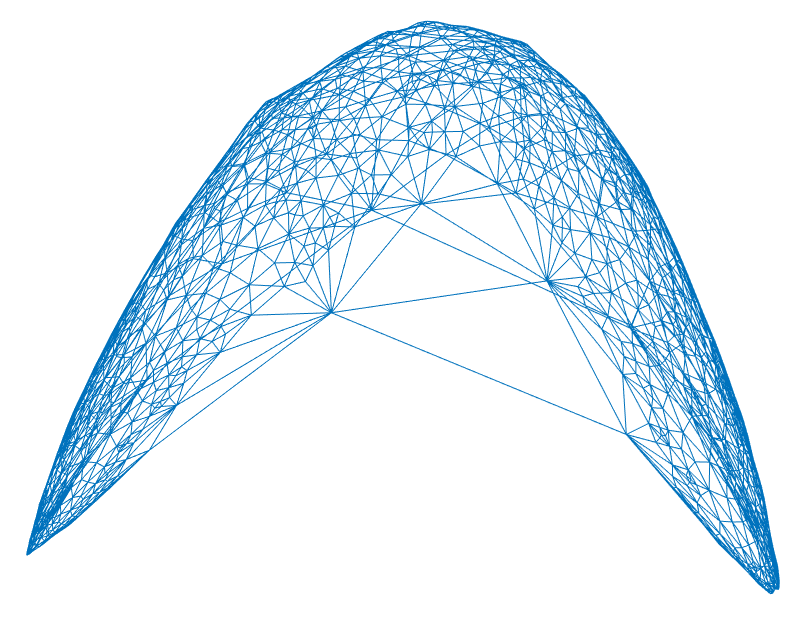}} \\
		\subfloat[Schur Complement Layout]{\includegraphics*[width=1.75 in,height = 1.75in]{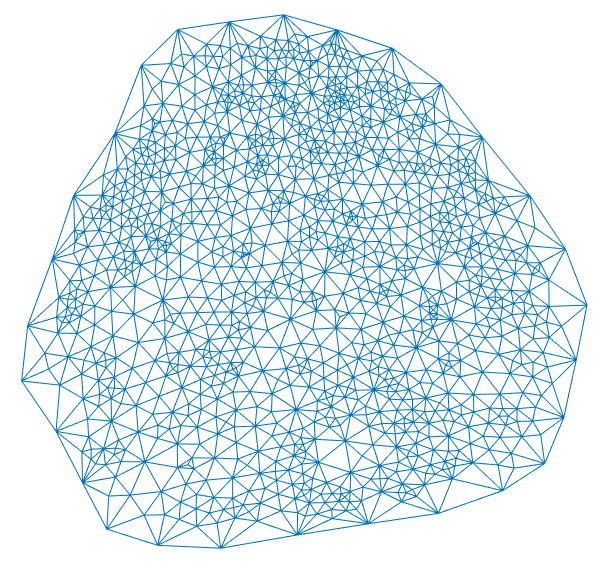}} \qquad
		\subfloat[Schur Complement Layout]{\includegraphics*[width= 3.5 in, height = 1.5 in]{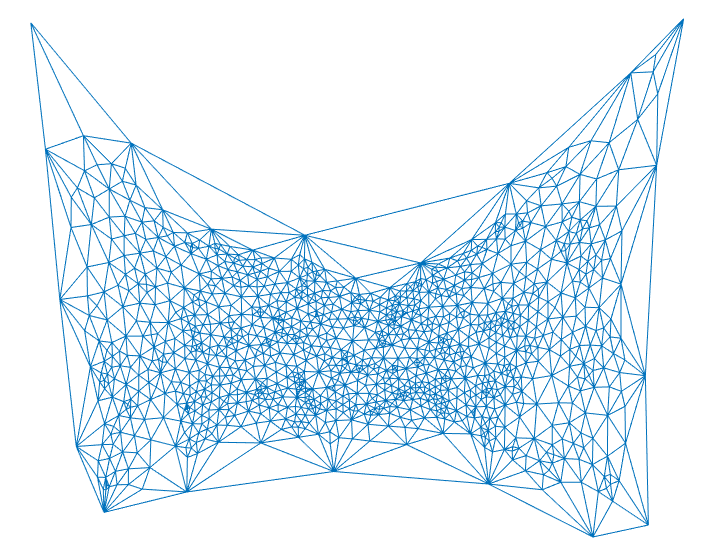}}
		\caption{Delaunay triangulations of $1250$ points randomly generated on the disk (A) and rectangle (B), their non-planar spectral layouts (C) and (D), and planar layouts using a spring embedding of the Schur complement of the graph Laplacian with respect to the interior vertices (E) and (F).}
		\label{fig1}
	\end{center}
\end{figure}

In this work, we examine how to embed the boundary face such that the embedding is convex and minimizes Hall's energy over all such convex embeddings with some given normalization. While it is not clear how to exactly minimize energy over all convex embeddings in polynomial time, it also is not clear that this is a NP-hard optimization problem. Proving that this optimization problem is NP-hard appears to be extremely difficult, as the problem itself seems to lack any natural relation to a known NP-complete problem. In what follows, we analyze this problem and produce an algorithm with theoretical guarantees for a large class of three-connected planar graphs.

\begin{figure}[t] 
	\begin{center}
		\subfloat[Laplacian Embedding]{\includegraphics*[width=3in]{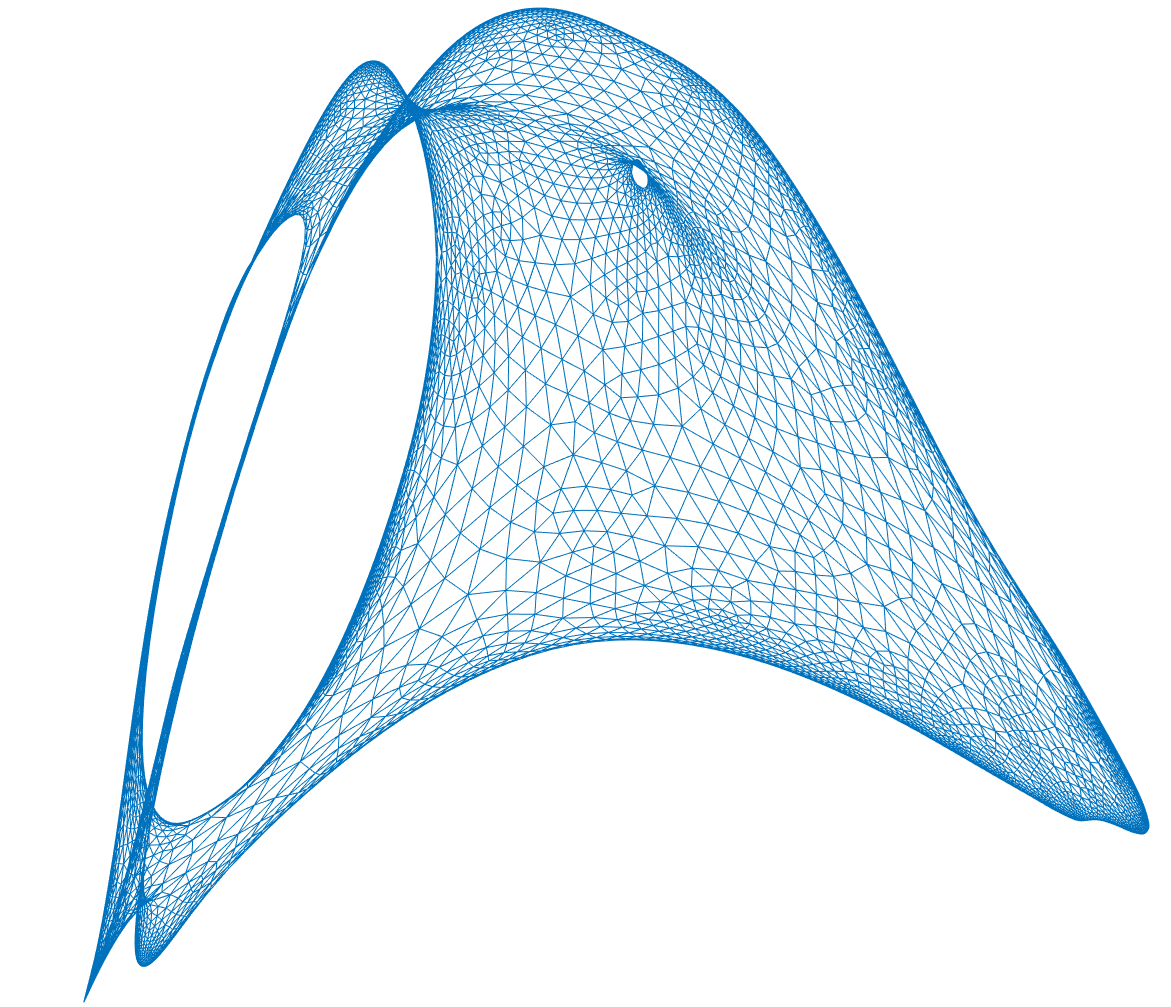}}
		 \qquad
		\subfloat[Schur Complement Embedding]{\includegraphics*[width=3 in]{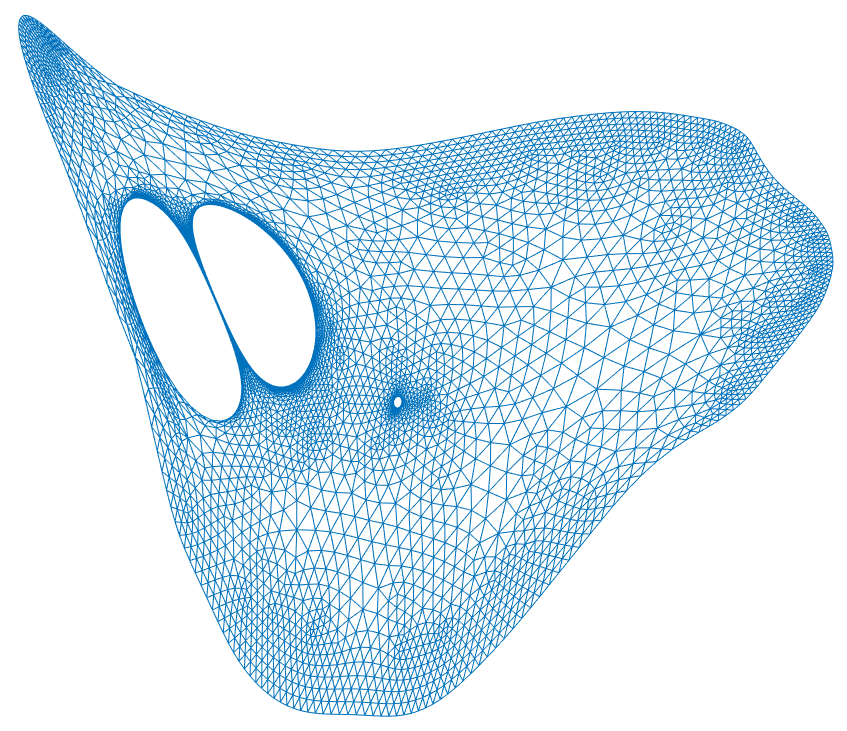}}
		\caption{A visual example of embeddings of the 2D finite element discretization graph 3elt, taken from the SuiteSparse Matrix Collection \cite{davis2011university}. Figure (A) is the non-planar spectral layout of this 2D mesh, and Figure (B) is a planar spring embedding of the mesh, using the minimal non-trivial eigenvectors of the Schur complement to embed the boundary. }
		\label{fig:3elt}
	\end{center}
\end{figure}

 Our analysis begins by observing that the Schur complement of the graph Laplacian with respect to the interior vertices is the correct matrix to consider when choosing an optimal embedding of boundary vertices. See Figure \ref{fig:3elt} for a visual example of a spring embedding using the two minimial non-trivial eigenvectors of the Schur complement. In order to theoretically understand the behavior of the Schur complement, we prove a discrete trace theorem. Trace theorems are a class of results in theory of partial differential equations relating norms on the domain to norms on the boundary, which are used to provide a priori estimates on the Dirichlet integral of functions with given data on the boundary. We construct a discrete version of a trace theorem in the plane for ``energy"-only semi-norms. Using a discrete trace theorem, we show that this Schur complement is spectrally equivalent to the boundary Laplacian to the one-half power. This spectral equivalence result produces theoretical guarantees for the energy minimizing spring embedding problem, but is also of independent interest and applicability in the study of spectral properties of planar graphs. These theoretical guarantees give rise to a natural algorithm with provable guarantees. The performance of this algorithm is also illustrated through numerical experiments.

The remainder of the paper is as follows. In Section 2, we formally introduce Tutte's spring embedding theorem, characterize the optimization problem under consideration, and illustrate the connection to the Schur complement. In Section 3, we consider trace theorems for Lipschitz domains from the theory of elliptic partial differential equations, prove discrete energy-only variants of these results for the plane, and show that the Schur complement with respect to the interior is spectrally equivalent to the boundary Laplacian to the one-half power. In Section 4, we use the results from the previous section to give theoretical guarantees regarding approximate solutions to the original optimization problem, and use these theoretical results to motivate an algorithm to embed the outer face of a spring embedding. We present numerical results to illustrate both the behavior of Schur complement-based embeddings compared to variations of natural spectral embeddings, and the practical performance of the algorithm introduced.

\section{Spring Embeddings and the Schur Complement} \label{sec:spring_schur}

In this section, we introduce the main definitions and notation of the paper, formally define the optimization problem under consideration, and show how the Schur complement is closely related to this optimization problem. 

\subsection{Definitions and Notation}

Let $G= (V,E)$, $V=\{1,...,n\} $, $E \subset \{ e \subset V \, | \, |e| = 2\}$, be a simple, connected, undirected graph. $G$ is $k$-connected if it remains connected upon the removal of any $k-1$ vertices, and is planar if it can be drawn in the plane such that no edges intersect (save for adjacent edges at their mutual endpoint). A face of a planar embedding of a graph is a region of the plane bounded by edges (including the outer infinite region, referred to as the outer face). Let $\mathcal{G}_n$ be the set of all ordered pairs $(G,\Gamma)$, where $G$ is a simple, undirected, planar, three-connected graph of order $n>4$, and $\Gamma \subset V$,  $n_\Gamma:=|\Gamma|$, are the vertices of some face of $G$. Three-connectedness is an important property for planar graphs, which, by Steinitz's theorem, guarantees that the graph is the skeleton of a convex polyhedron \cite{steinitz1922polyeder}. This characterization implies that for three-connected graphs ($n>4$), the edges corresponding to each face in a planar embedding are uniquely determined by the graph. In particular, the set of faces is simply the set of induced cycles, so we may refer to faces of the graph without specifying an embedding. One important corollary of this result is that, for $n>4$, the vertices of any face form an induced simple cycle. Let $N_G(i)$ be the neighborhood of vertex $i$, $N_G(S)$ be the union of the neighborhoods of the vertices in $S$, and $d_G(i,j)$ be the distance between vertices $i$ and $j$ in the graph $G$. When the associated graph is obvious, we may remove the subscript. Let $d(i)$ be the degree of vertex $i$. Let $G[S]$ be the graph induced by the vertices $S$, and $d_{S}(i,j)$ be the distance between vertices $i$ and $j$ in $G[S]$. If $H$ is a subgraph of $G$, we write $H \subset G$. The Cartesian product $G_1 \square G_2$ between $G_1 =(V_1,E_1)$ and $G_2 = (V_2,E_2)$, is the graph with vertices $(v_1,v_2) \in V_1 \times V_2$ and edges $ \left((u_1,u_2),(v_1,v_2)\right) \in E$ if $(u_1,v_1) \in E_1$ and $u_2 = v_2$, or $u_1 = v_1$ and $(u_2,v_2) \in E_2$. The graph Laplacian $L_G \in \R{n \times n}$ of $G$ is the symmetric matrix defined by
$$ \langle L_G x ,  x \rangle = \sum_{\{i,j\} \in E} (x_i - x_j)^2,$$
and, in general, a matrix is the graph Laplacian of some weighted graph if it is symmetric diagonally dominant, has non-positive off-diagonal entries, and the vector $\mathbf{1}:=(1,...,1)^T$ lies in its nullspace. The convex hull of a finite set of points $X$ is denoted by conv$(X)$, and a point $x \in X$ is a vertex of conv$(X)$ if $x \not \in \text{conv}(X \backslash x)$. Given a matrix $A$, we denote the $i^{th}$ row by $A_{i,\boldsymbol{\cdot}}$, the $j^{th}$ column by $A_{\boldsymbol{\cdot},j}$, and the entry in the $i^{th}$ row and $j^{th}$ column by $A_{i,j}$.

\subsection{Spring Embeddings}

Here and in what follows, we refer to $\Gamma$ as the ``boundary" of the graph $G$, $V\backslash \Gamma$ as the ``interior," and generally assume $n_\Gamma:= |\Gamma|$ to be relatively large (typically $n_\Gamma = \Theta(n^{1/2})$). Of course, the concept of a ``boundary" face is somewhat arbitrary, though, depending on the application from which the graph originated (i.e., a discretization of some domain), one face is often already designated as the boundary face. If a face has not been designated, choosing the largest induced cycle is a reasonable choice.  By embedding $G$ in the plane and traversing the embedding, one can easily find all the induced cycles of $G$ in linear time and space \cite{chiba1985linear}.

Without loss of generality, suppose that $\Gamma = \{n- n_\Gamma + 1,..., n\}$. A matrix $X \in \R{n \times 2}$ is said to be a planar embedding of $G$ if the drawing of $G$ using straight lines and with vertex $i$ located at coordinates $X_{i,\boldsymbol{\cdot}}$ for all $i$ is a planar drawing. A matrix $X_\Gamma \in \R{n_\Gamma \times 2}$ is said to be a convex embedding of $\Gamma$ if the embedding is planar and every point is a vertex of the convex hull $\text{conv}(\{[X_\Gamma]_{i,\boldsymbol{\cdot}}\}_{i=1}^{n_\Gamma})$. Tutte's spring embedding theorem states that if $X_\Gamma$ is a convex embedding of $\Gamma$, then the system of equations
$$ X_{i,\boldsymbol{\cdot}} = \begin{cases} \frac{1}{d(i)} \sum_{j \in N(i)} X_{j,\boldsymbol{\cdot}} \quad \; \; i = 1,...,n-n_\Gamma \\
[X_{\Gamma}]_{i-(n-n_\Gamma), \boldsymbol{\cdot}} \qquad i = n-n_\Gamma+1 ,...,n \end{cases}$$
has a unique solution $X$, and this solution is a planar embedding of $G$ \cite{MR0158387}. 

We can write both the Laplacian and embedding of $G$ in block-notation,
differentiating between interior and boundary vertices as follows:
$$L_{G} = \begin{pmatrix} L_{o} + D_{o} & -A_{o,\Gamma} \\ -A_{o,\Gamma}^T &
L_{\Gamma} + D_\Gamma \end{pmatrix} \in \R{n \times n},\quad X = \begin{pmatrix} X_o \\ X_\Gamma \end{pmatrix} \in \R{n \times 2},$$
where $L_{o},D_o \in \R{(n - n_\Gamma ) \times (n - n_\Gamma )}$, $L_\Gamma, D_\Gamma \in \R{ n_\Gamma  \times n_\Gamma }$, $A_{o,\Gamma} \in \R{(n- n_\Gamma ) \times n_\Gamma }$, $X_o \in \R{(n- n_\Gamma ) \times 2}$, $X_\Gamma \in \R{ n_\Gamma  \times 2}$, and $L_o$ and $L_\Gamma$ are the Laplacians of $G[V\backslash \Gamma]$ and $G[\Gamma]$, respectively. Using block notation, the system of equations for the Tutte spring embedding of some convex embedding $X_\Gamma$ is given by
$$ X_o = ( D_o + D[L_o] )^{-1} [(D[L_o] - L_o)X_o + A_{o,\Gamma} X_\Gamma],$$
where $D[A]$ is the diagonal matrix with diagonal entries given by the diagonal of $A$. Therefore, the unique solution to this system is
$$X_o =(L_o +D_o)^{-1} A_{o,\Gamma} X_\Gamma.$$
We note that this choice of $X_o$ not only guarantees a planar embedding of $G$, but also minimizes Hall's energy, namely,
	$$\arg\min_{X_o} h(X) = (L_o +D_o)^{-1} A_{o,\Gamma} X_\Gamma,$$
	where $h(X):= \text{Tr}(X^T L X)$ (see \cite{MR2029596} for more on Hall's energy).
	
While Tutte's theorem is a very powerful result, guaranteeing that, given a convex embedding of any face, the energy minimizing embedding of the remaining vertices results in a planar embedding, it gives no direction as to how this outer face should be embedded. In this work, we consider the problem of producing a planar embedding that is energy minimizing, subject to some normalization.  We consider embeddings that satisfy $X_\Gamma^T X_\Gamma = I$ and $X^T_\Gamma \mathbf{1} = 0$, though other normalizations, such as $X^T X = I$ and $X^T \mathbf{1} = 0$, would be equally appropriate. The analysis that follows in this paper can be readily applied to this alternate normalization, but it does require the additional step of verifying a norm equivalence between $V$ and $\Gamma$ for the harmonic extension of low energy vectors, which can be produced relatively easily for the class of graphs considered in Section \ref{sec:trace}. Let $\mathcal{X}$ be the set of all convex, planar embeddings $X_\Gamma$ that satisfy $X_\Gamma^T X_\Gamma = I$ and $X^T_\Gamma \mathbf{1} = 0$. The main optimization problem under consideration is

\begin{equation} 
\min \; h(X) \quad s.t. \quad  X_\Gamma \in \text{cl}(\mathcal{X}),\label{eqn:opt}
\end{equation}

where cl$(\boldsymbol{\cdot})$ is the closure of a set. $\mathcal{X}$ is not a closed set, and so the minimizer of (\ref{eqn:opt}) may be a non-convex embedding. However, by the definition of closure, any such minimizer is arbitrarily close to a convex embedding. The normalizations $X^T_\Gamma \mathbf{1} = 0$ and $X^T_\Gamma X_\Gamma = I$ ensure that the solution does not degenerate into a single point or line. In what follows we are primarily concerned with approximately solving this optimization problem. It is unclear whether there exists an efficient algorithm to solve (\ref{eqn:opt}) or if the associated decision problem is NP-hard. If (\ref{eqn:opt}) is NP-hard, it seems extremely difficult to verify that this is indeed the case. This remains an open problem.

\subsection{Schur Complement of $V \backslash \Gamma$}

Given some choice of $X_\Gamma$, by Tutte's theorem the minimum value of $h(X)$ is attained when $X_o =(L_o +D_o)^{-1} A_{o,\Gamma} X_\Gamma$, and given by
\begin{eqnarray*}
h(X) &=& \text{Tr} \left[ \begin{pmatrix} [(L_o +D_o)^{-1} A_{o,\Gamma}  X_\Gamma]^T & X^T_\Gamma \end{pmatrix} \begin{pmatrix} L_{o} + D_{o} & -A_{o,\Gamma} \\ -A_{o,\Gamma}^T &
L_{\Gamma} + D_\Gamma  \end{pmatrix} \begin{pmatrix} (L_o +D_o)^{-1} A_{o,\Gamma}  X_\Gamma \\ X_\Gamma \end{pmatrix} \right] \\
&=& \text{Tr} \big(X_\Gamma^T \big[  L_\Gamma + D_\Gamma - A^T_{o,\Gamma} ( L_o + D_o)^{-1} A_{o,\Gamma}  \big] X_\Gamma \big) \, = \, \text{Tr} \left(X_\Gamma^T S_\Gamma X_\Gamma \right),
\end{eqnarray*}
where $S_\Gamma$ is the Schur complement of $L_G$ with respect to $V \backslash \Gamma$,
$$S_\Gamma =  L_\Gamma + D_\Gamma - A^T_{o,\Gamma} ( L_o + D_o)^{-1} A_{o,\Gamma}. $$
For this reason, we can treat $X_o$ as a function of $X_\Gamma$ and instead consider the optimization problem
\begin{equation} 
\min \; h_\Gamma(X_\Gamma) \quad s.t. \quad  X_\Gamma \in \text{cl}(\mathcal{X}),\label{p1}
\end{equation}
where
$$ h_\Gamma(X_\Gamma):=\text{Tr} \big(X_\Gamma^T S_\Gamma X_\Gamma \big) .$$

This immediately implies that, if the minimal two non-trivial eigenvectors of $S_\Gamma$ produce a convex embedding, then this is the exact solution of (\ref{p1}). However, a priori, there is no reason to think that this embedding would be planar or convex. In Section \ref{sec:energy_min_embed}, we perform numerical experiments that suggest that this embedding is often planar, and ``near" a convex embedding in some sense. However, even if the embedding is planar, converting a non-convex embedding to a convex one may increase the objective function by a large amount. In Section \ref{sec:trace}, we show that $S_\Gamma$ and $L_\Gamma^{1/2}$ are spectrally equivalent. This spectral equivalence leads to provable guarantees for an algorithm to approximately solve (\ref{p1}), as the minimal two eigenvectors of $L_\Gamma^{1/2}$ are planar and convex.

First, we present a number of basic properties of the Schur complement of a graph Laplacian in the following proposition. For more information on the Schur complement, we refer the reader to \cite{carlson1986schur,MR625249,MR2160825}.

\begin{proposition}\label{prop:schur}
	Let $G=(V,E)$, $n = |V|$, be a graph and $L_G \in \R{n \times n}$ the associated graph Laplacian. Let $L_G$ and vectors $v \in \R{n}$ be written in block form
	\begin{equation*}  L(G) = \begin{pmatrix}  L_{11} & L_{12} \\ L_{21} &
	L_{22} \end{pmatrix} ,\quad v = \begin{pmatrix} v_1 \\ v_2 \end{pmatrix},\quad \end{equation*}
	where $L_{22} \in \R{m \times m}$, $v_2 \in \R{m}$, and $L_{12} \ne 0$. Then
	\begin{enumerate}[(1)]
		\item $S=L_{22} - L_{21} L^{-1}_{11} L_{12}$ is a graph Laplacian,
		\item $ \sum_{i=1}^m (e^T_i L_{22} \mathbf{1}_m)  e_i e^T_i    - L_{21} L^{-1}_{11} L_{12}$ is a graph Laplacian,
		\item $\langle S w , w \rangle = \inf \{ \langle L v , v \rangle  | v_2 = w \}$.
	\end{enumerate}
\end{proposition}

\begin{proof}
	Let $P = \begin{pmatrix} -L_{11}^{-1} L_{12} \\  I \end{pmatrix} \in \R{n \times m}$. Then
$$ P^T L P =  \begin{pmatrix} -L_{21} L_{11}^{-1}  &  I \end{pmatrix} \begin{pmatrix}  L_{11} & L_{12} \\ L_{21} &
	L_{22} \end{pmatrix}  \begin{pmatrix} -L_{11}^{-1} L_{12} \\  I \end{pmatrix} = L_{22} - L_{21} L^{-1}_{11} L_{12} = S.$$
	Because $ L_{11} \mathbf{1}_{n-m} + L_{12} \mathbf{1}_m = 0$, we have $\mathbf{1}_{n-m} = -L_{11}^{-1} L_{12} \mathbf{1}_m$. Therefore $P \mathbf{1}_m = \mathbf{1}_n$, and, as a result,
	\begin{equation*}S \mathbf{1}_m = P^T L P \mathbf{1}_m = P^T L \mathbf{1}_n = P^T 0 = 0.\end{equation*}
	In addition,
	\begin{eqnarray*}
		\left[\sum_{i=1}^m (e^T_i L_{22} \mathbf{1}_m)  e_i e^T_i    - L_{21} L^{-1}_{11} L_{12} \right] \mathbf{1}_m &=& \bigg[ \sum_{i=1}^m (e^T_i L_{22} \mathbf{1}_m)  e_i e^T_i - L_{22} \bigg] \mathbf{1}_m  + S \mathbf{1}_m   \\ 
		&=& \sum_{i=1}^m (e^T_i L_{22} \mathbf{1}_m)  e_i - L_{22} \mathbf{1}_m \\ &=& \bigg[ \sum_{i=1}^m   e_i e^T_i  - I_m \bigg] L_{22} \mathbf{1}_m \, = \, 0.
	\end{eqnarray*}
	$L_{11}$ is an M-matrix, so $L_{11}^{-1}$ is a non-negative matrix. $L_{21} L_{11}^{-1} L_{12}$ is the product of three non-negative matrices, and so must also be non-negative. Therefore, the off-diagonal entries of $S$ and $ \sum_{i=1}^m (e^T_i L_{22} \mathbf{1})  e_i e^T_i    - L_{21} L^{-1}_{11} L_{12}$ are non-positive, and so both are graph Laplacians.
	
	Consider
	$$\langle L v, v \rangle = \langle L_{11} v_1 , v_1 \rangle + 2 \langle L_{12} v_2 , v_1 \rangle  + \langle L_{22} v_2, v_2 \rangle ,$$
	 with $v_2$ fixed. Because $L_{11}$ is symmetric positive definite, the minimum occurs when
	$$ \frac{\partial} { \partial v_1} \langle L v, v \rangle = 2 L_{11} v_1 + 2 L_{12} v_2 = 0.$$
	Setting $v_1 = -L_{11}^{-1} L_{12} v_2$, the desired result follows.
\end{proof}

The Schur complement Laplacian $S_\Gamma$ is the sum of two
Laplacians $L_\Gamma$ and $D_\Gamma - A^T_{o,\Gamma} ( L_o + D_o)^{-1} A_{o,\Gamma}$, where the first is the Laplacian of $G[\Gamma]$, and the second is a Laplacian representing
the dynamics of the interior. 

In the next section we prove the spectral equivalence of $S_\Gamma$ and $L_\Gamma^{1/2}$ for a large class of graphs by first proving discrete energy-only trace theorems. Then, in Section \ref{sec:energy_min_embed}, we use this spectral equivalence to prove theoretical properties of (\ref{p1}) and motivate an algorithm to approximately solve this optimization problem.

\section{Trace Theorems for Planar Graphs} \label{sec:trace}
The main result of this section takes classical trace theorems from the theory of partial differential equations and extends them to a class of planar graphs. However, for our purposes, we require a stronger form of trace theorem, one between energy semi-norms (i.e., no $\ell^2$ term), which we refer to as ``energy-only" trace theorems. These energy-only trace theorems imply their classical variants with $\ell^2$ terms almost immediately. We then use these new results to prove the spectral equivalence of $S_\Gamma$ and $L_\Gamma^{1/2}$ for the class of graphs under consideration. This class of graphs is rigorously defined below, but includes planar three-connected graphs that have some regular structure (such as graphs of finite element discretizations). In what follows, we prove spectral equivalence with explicit constants. While this does make the analysis slightly messier, it has the benefit of showing that equivalence holds for constants that are not too large, thereby verifying that the equivalence is a practical result which can be used in the analysis of algorithms. We begin by formally describing a classical trace theorem.

Let $\Omega \subset \mathbb{R}^d$ be a domain with boundary $\Gamma = \delta \Omega$ that, locally, is a graph of a
Lipschitz function. $H^1(\Omega)$ is the Sobolev space of
square integrable functions with square integrable weak gradient, with norm
\[
\| u \|^2_{1,\Omega } = \|\nabla u \|^2_{L^2(\Omega) } + \| u \|^2_{L^2(\Omega)},
\quad \mbox{where}\quad \| u \|^2_{L_2(\Omega)} = \int_\Omega u^2 \, dx .
\]
Let
\[ 
\| \varphi \|^2_{1/2,\Gamma } = \| \varphi \|^2_{L_2(\Gamma)} + 
\iint_{\Gamma\times \Gamma}\frac{(\varphi(x) - \varphi(y))^2}{|x-y|^d} \,dx \, dy
\]
for functions defined on $\Gamma$, and denote by $H^{1/2}(\Gamma)$ the
Sobolev space of functions defined on the boundary $\Gamma$ for which
$\|\boldsymbol{\cdot}\|_{1/2,\Gamma}$ is finite. The trace theorem for
functions in $H^1(\Omega)$ is one of the most important and used trace
theorems in the theory of partial differential equations.  More
general results for traces on boundaries of Lipschitz domains, which
involve $L^p$ norms and fractional derivatives, are due
E.~Gagliardo~\cite{1957GagliardoE-aa} (see
also~\cite{1988CostabelM-aa}). Gagliardo's theorem, when applied to
the case of $H^1(\Omega)$ and $H^{1/2}(\Gamma)$, states that if
$\Omega\subset\mathbb{R}^d$ is a Lipschitz domain, then the norm
equivalence
		\[
		\|\varphi\|_{1/2,\Gamma} \eqsim \inf\{ \|u\|_{1,\Omega} \;\big|\;  u|_\Gamma = \varphi\}
		\]
holds (the right hand side is indeed a norm on $H^{1/2}(\Gamma)$). These results are key tools in proving a priori estimates on
the Dirichlet integral of functions with given data on the boundary of
a domain $\Omega$.  Roughly speaking, a trace theorem gives a bound on
the energy of a harmonic function via norm of the trace of the
function on $\Gamma=\partial\Omega$. In addition to the classical
references given above, further details on trace theorems and their
role in the analysis of PDEs (including the case of Lipschitz domains)
can be found in \cite{1972LionsJ_MagenesE-aa,1967NecasJ-aa}. There are several analogues of this
theorem for finite element spaces (finite dimensional subspaces of
$H^1(\Omega)$). For instance, in~\cite{MR1126677} it is shown that the
finite element discretization of the Laplace-Beltrami operator on the
boundary to the one-half power provides a norm which is equivalent to the
$H^{1/2}(\Gamma)$-norm. Here we prove energy-only analogues of the
classical trace theorem for graphs $(G, \Gamma) \in \mathcal{G}_n$,
using energy semi-norms
$$ | u |^2_{G} = \langle L_G u,  u \rangle \qquad \text{and} \qquad | \varphi|^2_{\Gamma} =  \sum_{\substack{p,q \in \Gamma, \\p<q}} \frac{\left(\varphi(p)-\varphi(q)\right)^2}{d^2_G(p,q)}.$$

The energy semi-norm $| \boldsymbol{\cdot} |_{G}$ is a discrete analogue of $\|\nabla u \|_{L^2(\Omega) }$, and the boundary semi-norm $| \boldsymbol{\cdot} |_{\Gamma}$ is a discrete analogue of the quantity $\iint_{\Gamma\times \Gamma} \frac{(\varphi(x) - \varphi(y))^2}{|x-y|^{2}} \, dx \, dy$. In addition, by connectivity, $| \boldsymbol{\cdot} |_{G}$ and $| \boldsymbol{\cdot} |_{\Gamma}$ are norms on the quotient space orthogonal to $\c1$. We aim to prove that for any $\varphi \in \R{n_\Gamma}$,
$$ \frac{1}{c_1} \, |\varphi|_\Gamma \le \min_{u|_{\Gamma} = \varphi} |u|_G \le c_2 \, |\varphi|_\Gamma$$
for some constants $c_1,c_2$ that do not depend on $n_\Gamma,n$. We begin by proving these results for a simple class of graphs, and then extend our analysis to more general graphs. Some of the proofs of the below results are rather technical, and are therefore reserved for the appendix.

\subsection{Trace Theorems for a Simple Class of Graphs}
Let $G_{k,\ell} = C_k \, \square \, P_\ell$ be the Cartesian product of the $k$ vertex cycle $C_k$ and the $\ell$ vertex path $P_{\ell}$, where $ 4 \ell <k < 2 c \ell$ for some constant $c \in \mathbb{N}$. The lower bound $4 \ell < k$ is arbitrary in some sense, but is natural, given that the ratio of boundary length to in-radius of a convex region is at least $2\pi$. Vertex $(i,j)$ in $G_{k,\ell}$ corresponds to the product of $i \in C_{k}$ and $j \in P_{\ell}$, $i = 1,...,k$, $j = 1,...,\ell$. The boundary of $G_{k,\ell}$ is defined to be $\Gamma = \{(i,1)\}_{i=1}^k$. Let $u \in \mathbb{R}^{k \times \ell}$ and $\varphi \in \mathbb{R}^k$ be functions on $G_{k,\ell}$ and $\Gamma$, respectively, with $u[(i,j)]$ denoted by $u(i,j)$ and $\varphi[(i,1)]$ denoted by $\varphi(i)$. For the remainder of the section, we consider the natural periodic extension of the vertices $(i,j)$ and the functions $u(i,j)$ and $\varphi(i)$ to the indices $i \in \mathbb{Z}$. In particular, if $i \not \in \{1,...,k\}$, then $(i,j) := (i^*,j)$, $\varphi(i) := \varphi( i^*)$, and $u(i,j) := u(i^*,j)$, where $i^* \in \{1,...,k\}$ and $ i^* = i \mod k$. Let $G^*_{k,\ell}$ be the graph resulting from adding to $G_{k,\ell}$ all edges of the form $\{(i,j),(i-1,j+1)\}$ and $\{(i,j),(i+1,j+1)\}$, $i = 1,...,k$, $j=1,...,\ell-1$. We provide a visual example of $G_{k,\ell}$ and $G^*_{k,\ell}$ in Figure \ref{fig:gkl}. First, we prove a trace theorem for $G_{k,\ell}$.

We have broken the proof of the trace theorem into two lemmas. Lemma \ref{lm:bounded} shows that the discrete trace operator is bounded, and Lemma \ref{lm:inverse} shows that it has a continuous right inverse. Taken together, these lemmas imply our desired result.

\begin{lemma}\label{lm:bounded} Let $G = G_{k,\ell}$, $ 4 \ell <k < 2 c \ell$, $c \in \mathbb{N}$, with boundary $\Gamma = \{(i,1)\}_{i=1}^k$. For any $u \in \mathbb{R}^{k \times \ell}$, the vector $\varphi = u|_\Gamma$ satisfies
	$$|\varphi|_{\Gamma} \le  \max\{ \sqrt{3c},2 \pi\} \, |u|_G.$$
\end{lemma}
\begin{proof}
	We can decompose $\varphi(p+h) - \varphi(h)$ into a sum of differences, given by
	\begin{eqnarray*}
	    \varphi(p+h)-\varphi(p) &=&
		\sum_{i=1}^{s-1} u(p+h,i)-u(p+h,i+1)\\
		&& + 
		\sum_{i=1}^{h} u(p+i,s)-u(p+i-1,s)\\
		&& + 
		\sum_{i=1}^{s-1} u(p,s-i+1) -u(p,s-i),
	\end{eqnarray*}
	where $s = \bigg\lceil \frac{ h }{c} \bigg \rceil$. By Cauchy-Schwarz,
	\begin{eqnarray*}
		 \sum_{p=1}^k \sum_{h=1}^{\lfloor k/2 \rfloor}  \left( \frac{\varphi(p+h)-\varphi(p)}{h} \right)^2 & \le & 
		3 \sum_{p=1}^k \sum_{h=1}^{\lfloor k/2 \rfloor} \left(
		\frac{1}{ h } \sum_{i=1}^{s-1} u(p+h,i)-u(p+h,i+1)
		\right)^2\\
		&& + 
		3 \sum_{p=1}^k \sum_{h=1}^{\lfloor k/2 \rfloor} \left(
		\frac{1}{ h } 		\sum_{i=1}^{h} u(p+i,s)-u(p+i-1,s)
		\right)^2\\
		&& + 3 \sum_{p=1}^k  \sum_{h=1}^{\lfloor k/2 \rfloor} \left( \frac{1}{ h }
		\sum_{i=1}^{s-1} u(p,s-i+1) -u(p,s-i)
		\right)^2.\\
	\end{eqnarray*}
	
	We bound the first and the second term separately. The third term is identical to the first. Using Hardy's inequality~\cite[Theorem~326]{HardyLittlewoodPolya}, we can bound the first term by
	\begin{eqnarray*}
	\sum_{p=1}^k \sum_{h=1}^{\lfloor k/2 \rfloor} \left(
		\frac{1}{ h } \sum_{i=1}^{s-1} u(p,i)-u(p,i+1)
		\right)^2 &=& \sum_{p=1}^k \sum_{s=1}^{\ell}  \left(
		\frac{1}{ s } \sum_{i=1}^{s-1} u(p,i)-u(p,i+1)
		\right)^2 \sum_{\substack{h:\lceil h/c \rceil = s \\ 1 \le h \le \lfloor k/2 \rfloor}} \frac{s^2}{h^2} \\
		&\le& 4 \sum_{p=1}^k \sum_{s=1}^{\ell-1}  \big( u(p,s)-u(p,s+1) \big)^2 \sum_{\substack{h:\lceil h/c \rceil = s \\ 1 \le h \le \lfloor k/2 \rfloor}} \frac{s^2}{h^2}.
	\end{eqnarray*}
We have
$$\sum_{\substack{h:\lceil h/c \rceil = s \\ 1 \le h \le \lfloor k/2 \rfloor}} \frac{s^2}{h^2} \le s^2 \sum_{i=c(s-1)+1}^{cs} \frac{1}{i^2} \le \frac{s^2(c-1)}{(c(s-1)+1)^2} \le \frac{4(c-1)}{(c+1)^2} \le \frac{1}{2}$$
for $s \ge 2$ ($c \ge 3$, by definition), and for $s = 1$,
$$ \sum_{\substack{h:\lceil h/c \rceil = 1 \\ 1 \le h \le \lfloor k/2 \rfloor}} \frac{1}{h^2}  \le \sum_{i=1}^\infty \frac{1}{i^2} = \frac{\pi^2}{6}.$$
Therefore, we can bound the first term by
$$ \sum_{p=1}^k \sum_{h=1}^{\lfloor k/2 \rfloor} \left(
\frac{1}{ h } \sum_{i=1}^{s-1} u(p,i)-u(p,i+1)
\right)^2  \le \frac{2 \pi^2}{3} \sum_{p=1}^k \sum_{s=1}^{\ell-1}  \big( u(p,s)-u(p,s+1) \big)^2.$$

 For the second term, we have
	\begin{eqnarray*}
	\sum_{p=1}^k \sum_{h=1}^{\lfloor k/2 \rfloor} \left(
		\frac{1}{ h } 		\sum_{i=1}^{h} u(p+i,s)-u(p+i-1,s)
		\right)^2 &\le&	\sum_{p=1}^k \sum_{h=1}^{\lfloor k/2 \rfloor} \frac{1}{ h } \sum_{i=1}^{h} \big( u(p+i,s)-u(p+i-1,s) \big)^2 \\
		&\le& c \sum_{p=1}^k \sum_{s=1}^{\ell} \big( u(p+1,s)-u(p,s) \big)^2.
	\end{eqnarray*}
	
	Combining these bounds produces the desired result
	$$| \varphi |_\Gamma \le \max\{ \sqrt{3c},2 \pi\}\, |u|_{G}.$$
	
	\end{proof}

\begin{figure}[t] 
	\begin{center}
		\subfloat[$G_{16,3} =C_{16} \square P_3$]{\includegraphics*[width=2in]{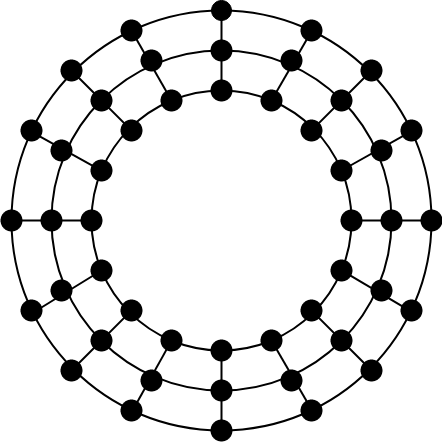}}
		\qquad \qquad
		\subfloat[$G^*_{16,3}$]{\includegraphics*[width=2 in]{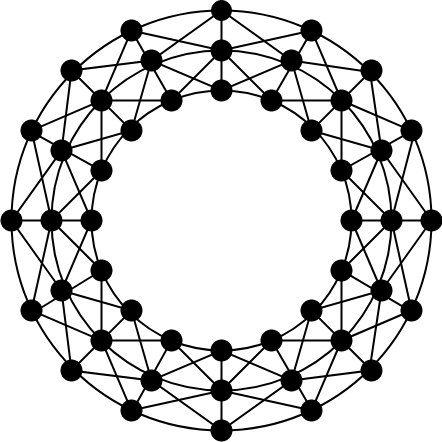}}
		\caption{A visual example of $G_{k,\ell}$ and $G^*_{k,\ell}$ for $k = 16$, $\ell = 3$. The boundary $\Gamma$ is given by the outer (or, by symmetry, inner) cycle.}
		\label{fig:gkl}
	\end{center}
\end{figure}

In order to show that the discrete trace operator has a continuous right inverse, we need to produce a provably low-energy extension of an arbitrary function on $\Gamma$. Let
	\begin{equation*}\label{ave}
	a=\frac{1}{k} \sum_{p = 1}^{k} \varphi(p) \qquad \text{and} \qquad 
	a(i,j) = \frac{1}{2j-1} \sum_{h=1-j}^{j-1} \varphi(i+h).
	\end{equation*}
We consider the extension
\begin{equation} \label{eqn:ext}
   u(i,j)= \frac{j-1}{\ell-1}a+\left(1-\frac{j-1}{\ell-1}\right)a(i,j).
\end{equation}
In the appendix (Lemma \ref{app:lm1}), we prove the following inverse result for the discrete trace operator.

\begin{lemma}\label{lm:inverse} Let $G =G_{k,\ell}$, $ 4 \ell <k < 2 c \ell$, $c \in \mathbb{N}$, with boundary $\Gamma = \{(i,1)\}_{i=1}^k$. For any $\varphi \in \R{k}$, the vector $u$ defined by (\ref{eqn:ext}) satisfies
	$$ |u|_G \le \sqrt{2c + \frac{233}{9}} \, |\varphi|_\Gamma.$$
\end{lemma}

Combining Lemmas \ref{lm:bounded} and \ref{lm:inverse}, we obtain our desired trace theorem.

\begin{theorem}\label{thm:disctrace} Let $G =G_{k,\ell}$, $ 4 \ell <k < 2 c \ell$, $c \in \mathbb{N}$, with boundary $\Gamma = \{(i,1)\}_{i=1}^k$. For any $\varphi \in \R{k}$,
$$ \frac{1}{\max\{ \sqrt{3c},2 \pi\}} \, |\varphi|_\Gamma \le \min_{u|_{\Gamma} = \varphi} |u|_G \le  \sqrt{2c + \frac{233}{9}} \, |\varphi|_\Gamma.$$
\end{theorem}

With a little more work, we can prove a similar result for a slightly more general class of graphs. Using Theorem \ref{thm:disctrace}, we can almost immediately prove a trace theorem for any graph $H$ satisfying $G_{k,\ell} \subset H \subset G^*_{k,\ell}$. In fact, Lemma \ref{lm:bounded} carries over immediately. In order to prove a new version of Lemma \ref{lm:inverse}, it suffices to bound the energy of $u$ on the edges in $G^*_{k,\ell}$ not contained in $G_{k,\ell}$. By Cauchy-Schwarz,
\begin{eqnarray*}
|u|_{G^*}^2 &=&  |u|^2_{G} + \sum_{i=1}^k \sum_{j=1}^{\ell-1} \bigg[ \left(u(i,j+1)-u(i-1,j) \right)^2 + \left(u(i,j+1)-u(i+1,j) \right)^2  \bigg] \\
&\le& 3 \sum_{i=1}^{k}\sum_{j=1}^{\ell}
    (u(i+1,j)-u(i,j))^2 +
    2 \sum_{i=1}^{k}\sum_{j=1}^{\ell-1}(u(i,j+1)-u(i,j))^2,
\end{eqnarray*}
and therefore Corollary \ref{thm:disctrace2} follows immediately from the proofs of Lemmas \ref{lm:bounded} and \ref{lm:inverse}.

\begin{corollary}\label{thm:disctrace2} Let $H$ satisfy $G_{k,\ell} \subset H \subset G^*_{k,\ell}$, $ 4 \ell <k < 2 c \ell$, $c \in \mathbb{N}$, with boundary $\Gamma = \{(i,1)\}_{i=1}^k$. For any $\varphi \in \R{k}$,
$$ \frac{1}{\max\{ \sqrt{3c},2 \pi\}} \, |\varphi|_\Gamma \le \min_{u|_{\Gamma} = \varphi} |u|_H \le \sqrt{4c + \frac{475}{9}} \, |\varphi|_\Gamma.$$
\end{corollary}

\subsection{Trace Theorems for General Graphs}

In order to extend Corollary \ref{thm:disctrace2} to more general graphs, we introduce a graph operation which is similar to in concept an aggregation (a partition of $V$ into connected subsets) in which the size of aggregates are bounded. In particular, we give the following definition.

\begin{definition}
The graph $H$, $G_{k,\ell}\subset H \subset G^*_{k,\ell}$, is said to be an $M$-aggregation of $(G,\Gamma) \in \mathcal{G}_n$ if there exists a partition $\mathcal{A} = a_* \cup \{ a_{i,j} \}_{i=1,...,k}^{j=1,...,\ell}$ of $V(G)$ satisfying
\begin{enumerate}
\item $G[a_{i,j}]$ is connected and $ |a_{i,j}| \le M$ for all $i =1,...,k$, $j=1,...,\ell$,
    \item $  \Gamma \subset \bigcup_{i=1}^k a_{i,1} $, and $\Gamma \cap a_{i,1} \ne \emptyset$ for all $i = 1,...,k$, 
    \item $ N_G(a_*) \subset a_* \cup \bigcup_{i=1}^k a_{i,\ell}$,
    \item the aggregation graph of $\mathcal{A}\backslash a_*$, given by $(\mathcal{A} \backslash a_*, \{ \left(a_{i_1,j_1}, a_{i_2,j_2} \right) \, | \, N_G(a_{i_1,j_1}) \cap a_{i_2,j_2} \ne 0\})$, is isomorphic to $H$.
\end{enumerate}
\end{definition}

We provide a visual example in Figure \ref{fig:magg}, and, later, in Subsection \ref{sub:ex}, we show that this operation applies to a fairly large class of graphs. For now, we focus using the above definition to prove trace theorems for graphs that have an $M$-aggregation $H$, for some $G_{k,\ell}\subset H \subset G^*_{k,\ell}$.

However, the $M$-aggregation procedure is not the only operation for which we can control the behavior of the energy and boundary semi-norms. For instance, the behavior of our semi-norms under the deletion of some number of edges can be bounded easily if there exists a set of paths of constant length, with one path between each pair of vertices which are no longer adjacent, such that no edge is in more than a constant number of these paths. In addition, the behavior of these semi-norms under the disaggregation of large degree vertices is also relatively well-behaved, see \cite{hu2017approximation} for details. We give the following result regarding graphs $(G,\Gamma)$ for which some $H$, $G_{k,\ell}\subset H \subset G^*_{k,\ell}$, is an $M$-aggregation of $(G,\Gamma)$, but note that a large number of minor refinements are possible, such as the two briefly mentioned in this paragraph.

\begin{figure}[t] 
	\begin{center}
		\subfloat[graph $(G,\Gamma)$]{\includegraphics*[width=1.9in]{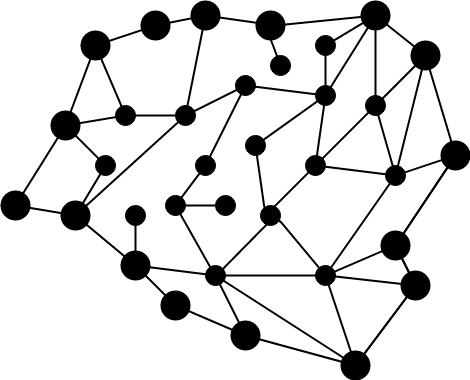}}
		\qquad
		\subfloat[partition $\mathcal{A}$]{\includegraphics*[width=1.9 in]{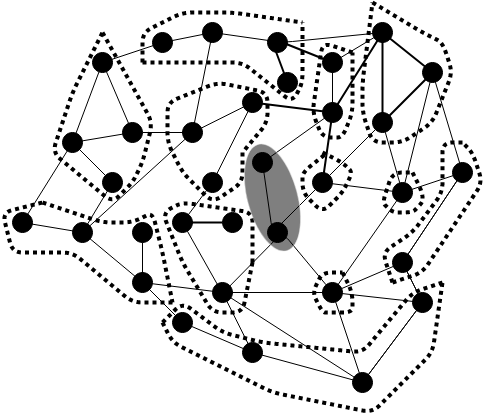}} \qquad	\subfloat[$G_{6,2} \subset H \subset G^*_{6,2}$]{\includegraphics*[width=1.6 in]{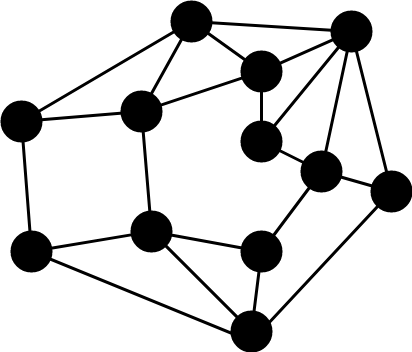}}
		\caption{An example of an $M$-aggregation. Figure (A) provides a visual representation of a graph $G$, with boundary vertices $\Gamma$ enlarged. Figure (B) shows a partition $\mathcal{A}$ of $G$, in which each aggregate (enclosed by dotted lines) has order at most four. The set $a_*$ is denoted by a shaded region. Figure (C) shows the aggregation graph $H$ of $\mathcal{A}\backslash a_*$. The graph $H$ satisfies $G_{6,2} \subset H \subset G^*_{6,2}$, and is therefore a $4$-aggregation of $(G,\Gamma)$.}
		\label{fig:magg}
	\end{center}
\end{figure}

\begin{theorem}\label{thm:general-trace}
    If $H$, $G_{k,\ell}\subset H \subset G^*_{k,\ell}$, $4\ell<k<2c\ell$, $c \in \mathbb{N}$, is an $M$-aggregation of $(G,\Gamma) \in \mathcal{G}_n$, then for any $\varphi \in \R{n_\Gamma}$,
$$ \frac{1}{ 6 M  \sqrt{M+3} \max\{ \sqrt{3c},2 \pi\}} \, |\varphi|_\Gamma \le \min_{u|_{\Gamma} = \varphi} |u|_G  \le 28 M^2 \sqrt{ 3 c + 20}  \, |\varphi|_\Gamma.$$
\end{theorem}

The proof of this result is rather technical, and can be found in the appendix (Theorem \ref{app:thm1}). The same proof of Theorem \ref{thm:general-trace} also immediately implies a similar result. Let $\widetilde L \in \R{n_\Gamma \times n_\Gamma}$ be the Laplacian of the complete graph on $\Gamma$ with weights $w(i,j)= d^{-2}_\Gamma(i,j)$. The same proof implies the following.

\begin{corollary}\label{thm:tildeversion}
    If $H$, $G_{k,\ell}\subset H \subset G^*_{k,\ell}$, $4\ell<k<2c\ell$, $c \in \mathbb{N}$, is an $M$-aggregation of  $(G,\Gamma) \in \mathcal{G}_n$, then for any $\varphi \in \R{n_\Gamma}$,
$$ \frac{1}{ 6 M  \sqrt{M+3} \max\{ \sqrt{3c},2 \pi\}}  \, \langle \widetilde L \varphi,  \varphi \rangle^{1/2} \le \min_{u|_{\Gamma} = \varphi} |u|_G  \le 28 M^2 \sqrt{ 3 c + 20}   \, \langle \widetilde L \varphi,  \varphi \rangle^{1/2}.$$
\end{corollary}

\subsection{Spectral Equivalence of $S_\Gamma$ and $L_\Gamma^{1/2}$} By Corollary \ref{thm:tildeversion},  and the property $\langle \varphi , S_\Gamma \varphi \rangle =\min_{u|_{\Gamma} = \varphi} |u|^2_G $ (see Proposition \ref{prop:schur}), in order to prove spectral equivalence between $S_\Gamma$ and $L_\Gamma^{1/2}$, it suffices to show that $L_\Gamma^{1/2}$ and $\widetilde L$ are spectrally equivalent. This can be done relatively easily, and leads to a proof of the main result of the section.

\begin{theorem} \label{thm:specequiv}
      If $H$, $G_{k,\ell}\subset H \subset G^*_{k,\ell}$, $4\ell<k<2c\ell$, $c \in \mathbb{N}$, is an $M$-aggregation of  $(G,\Gamma) \in \mathcal{G}_n$, then for any $\varphi \in \R{n_\Gamma}$,
$$ \frac{1}{36 M^2 (M+3) \max\{ 3c,4 \pi^2 \} \left( \frac{2}{3 \pi} + \frac{\sqrt{2}}{27} \right)  }  \, \langle L^{1/2}_\Gamma \varphi,  \varphi \rangle \le \langle  S_\Gamma \varphi, \varphi \rangle \le \frac{784 M^4(3c+20)}{\left( \frac{1}{2 \pi} - \frac{\sqrt{2}}{12} \right)} \, \langle L^{1/2}_\Gamma \varphi,  \varphi \rangle.$$
\end{theorem}

\begin{proof}
    Let $\phi(i,j) = \min\{i-j  \mod n_\Gamma, \; j-i \mod n_\Gamma\}$. $G[\Gamma]$ is a cycle, so $ \widetilde L(i,j) = - \phi(i,j)^{-2}$ for $i \ne j$. The spectral decomposition of $L_\Gamma$ is well known, namely,
	\begin{equation*} L_\Gamma =  \sum_{k=1}^{\big\lfloor \tfrac{n_\Gamma}{2} \big\rfloor } \lambda_k(L_\Gamma) \bigg[\frac{x_k x_k^T}{\|x_k\|^2} +\frac{y_k y_k^T}{\|y_k\|^2} \bigg],\end{equation*}
	where $\lambda_k(L_\Gamma) = 2 - 2 \cos\tfrac{2 \pi k }{n_\Gamma} $ and $x_k(j) = \sin  \tfrac{2 \pi k j }{n_\Gamma}$, $y_k(j) = \cos \tfrac{2 \pi k j }{n_\Gamma} $, $j = 1,...,n_\Gamma$. If $n_\Gamma$ is odd, then $\lambda_{(n_\Gamma-1)/2}$ has multiplicity two, but if $n_\Gamma$ is even, then $\lambda_{n_\Gamma/2}$ has only multiplicity one, as $x_{n_\Gamma/2} = 0$.
	If $k \ne n_\Gamma/2$, we have
	\begin{equation*} \| x_k \|^2 = \sum_{j=1}^{n_\Gamma} \sin^2 \bigg( \frac{2 \pi k j}{n_\Gamma}\bigg) = \frac{n_\Gamma}{2} - \frac{1}{2} \sum_{j=1}^{n_\Gamma} \cos \bigg(  \frac{4 \pi k j}{n_\Gamma} \bigg) =\frac{n_\Gamma}{2} - \frac{1}{4} \left[ \frac{ \sin ( 2 \pi k (2 + \frac{1}{n_\Gamma} ) ) }{ \sin  \frac{2 \pi k }{n_\Gamma}  } - 1 \right] = \frac{n_\Gamma}{2},\end{equation*}
	and so $\|y_k\|^2 = \frac{n_\Gamma}{2}$ as well. If $k = n_\Gamma/2$, then $\|y_k \|^2 = n_\Gamma$. If $n_\Gamma$ is odd,
	\begin{eqnarray*}
	L_\Gamma^{1/2}(i,j) & = & \frac{2 \sqrt{2}}{n_\Gamma} \sum_{k=1}^{\frac{n_\Gamma-1}{2}} \left[1 -  \cos \left(\frac{2 k  \pi  }{n_\Gamma}\right) \right]^{1/2} \left[ \sin \left( \frac{2 \pi k i }{n_\Gamma} \right) 
		\sin \left( \frac{2 \pi k j }{n_\Gamma} \right) - \cos \left( \frac{2 \pi k i }{n_\Gamma} \right) \cos \left( \frac{2 \pi k j}{n_\Gamma} \right) \right] \\
		& = & \frac{4}{n_\Gamma} \sum_{k=1}^{\frac{n_\Gamma-1}{2}} \sin \left( \frac{\pi}{2} \frac{2 k }{n_\Gamma} \right) \cos \left( \phi(i,j) \pi \frac{2k}{n_\Gamma}\right) \, = \, \frac{2}{n_\Gamma} \sum_{k=0}^{n_\Gamma} \sin \left( \frac{\pi}{2} \frac{2 k }{n_\Gamma} \right) \cos \left( \phi(i,j) \pi \frac{2k}{n_\Gamma}\right),
	\end{eqnarray*}
	and if $n_\Gamma$ is even,
		\begin{eqnarray*}
	L_\Gamma^{1/2}(i,j) & = & \frac{2}{n_\Gamma}(-1)^{i+j} + \frac{4}{n_\Gamma} \sum_{k=1}^{\frac{n_\Gamma}{2}-1} \sin \left( \frac{\pi}{2} \frac{2 k }{n_\Gamma} \right) \cos \left( \phi(i,j) \pi \frac{2k}{n_\Gamma}\right) \\
	& = & \frac{2}{n_\Gamma} \sum_{k=0}^{n_\Gamma} \sin \left( \frac{\pi}{2} \frac{2 k }{n_\Gamma} \right) \cos \left( \phi(i,j) \pi \frac{2k}{n_\Gamma}\right). 
	\end{eqnarray*}
	$L^{1/2}_{\Gamma}(i,j)$ is simply the trapezoid rule applied to the integral of 
    $\sin (\tfrac{\pi}{2} x) \cos( \phi(i,j) \pi x)$ on the interval $[0,2]$. Therefore,
$$ \bigg|L_\Gamma^{1/2}(i,j) + \frac{2}{\pi(4 \phi(i,j)^2 -1)} \bigg| = \bigg|  L_\Gamma^{1/2}(i,j) -  \int_{0}^2 \sin\left(\frac{\pi}{2} x\right) \cos \left(\phi(i,j) \pi x\right) dx \bigg| \le \frac{2}{3 n_\Gamma^2} ,$$
where we have used the fact that if $f \in C^2([a,b])$, then 
$$ \bigg| \int_a^b f(x) dx - \frac{f(a)+f(b)}{2}(b-a) \bigg| \le \frac{(b-a)^3}{12}   \max_{\xi \in [a,b]} |f''(\xi)|.$$
Noting that $n_\Gamma \ge 3$, it quickly follows that
	$$ \bigg( \frac{1}{2 \pi} - \frac{\sqrt{2}}{12} \bigg) \langle \tilde L \varphi, \varphi \rangle \le \langle L^{1/2}_{\Gamma} \varphi, \varphi \rangle \le \bigg( \frac{2}{3 \pi} + \frac{\sqrt{2}}{27} \bigg) \langle \tilde L \varphi, \varphi \rangle.$$
Combining this result with Corollary \ref{thm:tildeversion}, and noting that $\langle \varphi , S_\Gamma \varphi \rangle = |\widehat u|^2_G$, where $\widehat u$ is the harmonic extension of $\varphi$, we obtain the desired result
$$ \frac{1}{36 M^2 (M+3) \max\{ 3c,4 \pi^2 \} \left( \frac{2}{3 \pi} + \frac{\sqrt{2}}{27} \right)  } \, \langle L^{1/2}_\Gamma \varphi,  \varphi \rangle \le \langle S_\Gamma \varphi,  \varphi \rangle \le \frac{584 M^4(3c+14)}{\left( \frac{1}{2 \pi} - \frac{\sqrt{2}}{12} \right)} \, \langle L^{1/2}_\Gamma \varphi,  \varphi \rangle.$$
\end{proof}

\subsection{An Illustrative Example}\label{sub:ex} While the concept of a graph $(G,\Gamma)$ having some $H$, $G_{k,\ell} \subset H \subset G^*_{k,\ell}$, as an $M$-aggregation seems somewhat abstract, this simple formulation in itself is quite powerful. As an example, we illustrate that this implies a trace theorem (and, therefore, spectral equivalence) for all three-connected planar graphs with bounded face degree (number of edges in the associated induced cycle) and for which there exists a planar spring embedding with a convex hull that is not too thin (a bounded distance to Hausdorff distance ratio for the boundary with respect to some point in the convex hull) and satisfies bounded edge length and small angle conditions. Let $\mathcal{G}_n^{f \le c}$ be the elements of $(G,\Gamma) \in \mathcal{G}_n$ for which every face other than the outer face $\Gamma$ has at most $c$ edges. We prove the following theorem\footnote{The below theorem is shown for $\ell \le k$ to avoid certain trivial cases involving small $n$. The same theorem holds for $n$ sufficiently large and $4 \ell < k$, but it should also be noted that the entire analysis of this section also holds for $\ell \le k$, albeit with worse constants.} in the appendix (Theorem \ref{app:thm2}).

\begin{theorem}\label{thm:example}
    If there exists a planar spring embedding $X$ of $(G, \Gamma) \in \mathcal{G}_n^{f\le c_1}$ for which
    \begin{enumerate}[(1)]
        \item $K= \text{conv}\left(\{ [X_\Gamma]_{i,\boldsymbol{\cdot}} \}_{i=1}^{n_\Gamma} \right)$ satisfies  
        $$\sup_{u \in K} \inf_{v \in \partial K} \sup_{w \in \partial K} \frac{\|u - v \| }{ \|u - w \|} \ge c_2>0,$$
        \item $X$ satisfies
        $$ \max_{\substack{\{i_1,i_2\} \in E \\ \{j_1,j_2\} \in E}} \frac{ \|X_{i_1,\boldsymbol{\cdot}} - X_{i_2,\boldsymbol{\cdot}} \|}{ \|X_{j_1,\boldsymbol{\cdot}} - X_{j_2,\boldsymbol{\cdot}} \|} \le c_3 \quad \text{and} \quad 
        \min_{\substack{i \in V \\ j_1,j_2 \in N(i)}} \angle \, X_{j_1,\boldsymbol{\cdot}}  \, X_{i,\boldsymbol{\cdot}} \, X_{j_2,\boldsymbol{\cdot}} \ge c_4>0,$$
    \end{enumerate}
    then there exists an $H$, $G_{k,\ell}\subset H \subset G^*_{k,\ell}$, $\ell \le k<2c\ell$, $c \in \mathbb{N}$, such that $H$ is an $M$-aggregation of $(G,\Gamma)$ where $c$ and $M$ are constants that depend on $c_1$, $c_2$, $c_3$, and $c_4$.
\end{theorem}

\section{Approximately Energy Minimizing Embeddings}\label{sec:energy_min_embed}

In this section, we make use of the analysis of Section \ref{sec:trace} to give theoretical guarantees regarding approximate solutions to (\ref{p1}), which inspires the construction of a natural algorithm to approximately solve this optimization problem. In addition, we give numerical results for our algorithm. Though in the previous section we took great care to produce results with explicit constants for the purpose of illustrating practical usefulness, in what follows we simply suppose that we have the spectral equivalence
\begin{equation} \label{eqn:speceq}
    \frac{1}{c_1} \, \langle L^{1/2}_\Gamma x,  x \rangle \le \langle S_\Gamma x,  x \rangle \le c_2 \, \langle L^{1/2}_\Gamma x,  x \rangle,
\end{equation}
for all $x \in \R{n_\Gamma}$ and some constants $c_1$ and $c_2$ which are not too large and can be explicitly chosen based on the results of Section \ref{sec:trace}.

\subsection{Theoretical Guarantees}

Again, we note that if the minimal two non-trivial eigenvectors of $S_\Gamma$ produce a convex embedding, then this is the exact solution of (\ref{p1}). However, if this is not the case, then, by spectral equivalence, we can still make a number of statements. 

The convex embedding $X_C$ given by
$$[X_{C}]_{j,\boldsymbol{\cdot}} = \frac{2}{n_\Gamma} \bigg(\cos \frac{2 \pi j}{n_\Gamma}, \sin  \frac{2 \pi j}{n_{\Gamma}} \bigg), \qquad j = 1,...,n_\Gamma,$$
is the embedding of the two minimal non-trivial eigenvectors of $L_\Gamma^{1/2}$, and therefore,
\begin{equation} \label{eqn:approx}
    h_\Gamma(X_C) \le 4 c_2 \sin  \frac{\pi}{n_\Gamma}  \le c_1 c_2 \min_{ X_\Gamma \in\text{cl}\left( \mathcal{X}\right)} h_\Gamma (X_\Gamma),
\end{equation}
thereby producing a $c_1c_2$ approximation guarantee for (\ref{p1}).

In addition, we can guarantee that the optimal embedding is largely contained in the subspace corresponding to the $k$ minimal eigenvalues of $L^{1/2}_\Gamma$ when $k$ is a reasonably large constant. In particular, if $X^*_\Gamma$ minimizes (\ref{p1}), and $\Pi_{i}$ is the $\ell^2$-orthogonal projection onto the direct sum of the eigenvectors corresponding to the $i$ minimal non-trivial eigenvalues (counted with multiplicity) of $L^{1/2}_\Gamma$, then
\begin{eqnarray*}
h_\Gamma(X_\Gamma^*) &\ge& \text{Tr}\big( \left[ (I - \Pi_{2 i}) X_\Gamma^* \right]^T S_\Gamma (I - \Pi_{2 i}) X_\Gamma^* \big) \\ 
&\ge& \frac{1}{c_1} \text{Tr}\big( \left[ (I - \Pi_{2 i}) X_\Gamma^* \right]^T L^{1/2}_\Gamma (I - \Pi_{2 i}) X_\Gamma^* \big) \\ 
&\ge& \frac{2}{c_1} \sin \left( \frac{\pi(i+1)}{n_\Gamma} \right) \text{Tr}\big( \left[ (I - \Pi_{2 i}) X_\Gamma^* \right]^T (I - \Pi_{2 \ell}) X_\Gamma^* \big),
\end{eqnarray*}
and $h_\Gamma(X_\Gamma^*) \le h_\Gamma(X_C)$, which, by using the property $\tfrac{2x}{\pi} \le \sin x \le x$ for all $x \in \left[0,\tfrac{\pi}{2}\right]$, implies that
$$\text{Tr}\big( \left[ (I - \Pi_{2 i}) X_\Gamma^* \right]^T (I - \Pi_{2 i}) X_\Gamma^* \big) \le  \frac{2 c_1 c_2 \sin \left(\pi/n_\Gamma \right)}{\sin\left(\pi(i+1)/n_\Gamma \right)} \le  \frac{\pi c_1 c_2}{i+1}.$$

\subsection{Algorithmic Considerations} The theoretical analysis of Subsection 4.1 inspires a number of natural techniques to approximately solve (\ref{p1}), such as exhaustively searching the direct sum of some constant number of low energy eigenspaces of $S_\Gamma$. However, numerically, it appears that when the pair $(G,\Gamma)$ satisfies certain conditions, such as the conditions of Theorem \ref{thm:example}, the minimal non-trivial eigenvector pair often produces a convex embedding, and when it does not, the removal of some small number of boundary vertices produces a convex embedding. If the embedding is almost convex (i.e., convex after the removal of some small number of vertices), a convex embedding can be produced by simply moving these vertices so that they are on the boundary and between their two neighbors.

 Given an approximate solution to (\ref{p1}), one natural approach simply consists of iteratively applying a smoothing matrix, such as $d I - S_\Gamma$, $d > \rho(S_\Gamma)$, or the inverse $S_\Gamma^{-1}$ defined on the subspace $\{x \, | \, \langle x, {\bf 1} \rangle = 0 \}$, until the matrix $X_\Gamma$ is no longer a convex embedding. In fact, applying this procedure to $X_C$ immediately produces a technique that approximates the optimal solution within a factor of at least $c_1 c_2$, and possibly better given smoothing. In order to have the theoretical guarantees that result from using $X_C$, and benefit from the possibly nearly-convex Schur complement low energy embedding, we introduce Algorithm \ref{alg1}.

\begin{algorithm}[!tb]
	\caption{Embed the Boundary $\Gamma$ \label{alg1}}
	\begin{enumerate}
		\item[] $ X = \text{minimaleigenvectors}(G,\Gamma)$
		\item[] { \bf If} $\text{isplanar}(X) = 0$,
		\begin{enumerate}
			\item[] $X \leftarrow \left\{\frac{2}{n_\Gamma} \bigg(\cos \frac{2 \pi j}{n_\Gamma}, \sin  \frac{2 \pi j}{n_{\Gamma}} \bigg)\right\}_{i=1}^{n_\Gamma}$
		\end{enumerate}
		\item[] { \bf Else}
		\begin{enumerate}
			\item[] { \bf If} $\text{isconvex}(X) = 1$,
			\begin{enumerate}
				\item[] $X_{alg} = X$ 
				\item[] end Algorithm
			\end{enumerate}
			\item[]  { \bf Else}
			\begin{enumerate}
				\item[] $X \leftarrow \text{makeconvex}(X)$ 
				\item[] $X \leftarrow X - \frac{{\bf 1}_{n_\Gamma}{\bf 1}^T_{n_\Gamma} X}{n_\Gamma}  $
				\item[] solve $[X^T X] Q = Q \Lambda$, $Q$ orthogonal, $\Lambda$ diagonal
				\item[] $X \leftarrow X Q \Lambda^{-1/2}$
				\item[] { \bf If } $h_\Gamma(X) > h_\Gamma \left( \left\{\frac{2}{n_\Gamma} \bigg(\cos \frac{2 \pi j}{n_\Gamma}, \sin  \frac{2 \pi j}{n_{\Gamma}} \bigg)\right\}_{i=1}^{n_\Gamma}\right)$
				\begin{enumerate}
					\item[] $X \leftarrow \left\{\frac{2}{n_\Gamma} \bigg(\cos \frac{2 \pi j}{n_\Gamma}, \sin  \frac{2 \pi j}{n_{\Gamma}} \bigg)\right\}_{i=1}^{n_\Gamma}$
				\end{enumerate}
			\end{enumerate}
		\end{enumerate}
		\item[] $\text{gap} =1$
		\item[] { \bf While } $\text{gap} >0$,
		\begin{enumerate}
			\item[] $\widehat X \leftarrow \text{smooth} (X)$
			\item[] { \bf If} $\text{isplanar}(\widehat X) = 0$,
			\begin{enumerate}
				\item[] $\text{gap} \leftarrow -1$
			\end{enumerate}
			\item[] { \bf Else}
			\begin{enumerate}
				\item[] { \bf If} $\text{isconvex}(\widehat X) = 0$,
				\begin{enumerate}
					\item[] $\widehat X \leftarrow \text{makeconvex}(\widehat X)$ 
				\end{enumerate}
				\item[] $\widehat X \leftarrow \widehat X - \frac{{\bf 1}_{n_\Gamma}{\bf 1}^T_{n_\Gamma} \widehat X}{n_\Gamma}  $
				\item[] solve $[\widehat X^T \widehat X] Q = Q \Lambda$, $Q$ orthogonal, $\Lambda$ diagonal
				\item[] $\widehat X \leftarrow \widehat X Q \Lambda^{-1/2}$
				\item[] $\text{gap} \leftarrow h_\Gamma(X) - h_\Gamma(\widehat X)$
				\item[] { \bf If} $\text{gap} >0$
				\begin{enumerate}
					\item[] $X \leftarrow \widehat X$
				\end{enumerate}
			\end{enumerate}
		\end{enumerate}
		\item[] $X_{alg} = X$
	\end{enumerate}
\end{algorithm}

Algorithm \ref{alg1} takes a graph $(G,\Gamma) \in \mathcal{G}_n$ as input, and first computes the minimal two non-trivial eigenvectors of the Schur complement, denoted by $X$. If $X$ is planar and convex, the algorithm terminates and outputs $X$, as it has found the exact solution to (\ref{p1}). If $X$ is non-planar, then this embedding is replaced by $X_C$, the minimal two non-trivial eigenvectors of the boundary Laplacian to the one-half power. If $X$ is planar, but non-convex, then some procedure is applied to transform $X$ into a convex embedding. The embedding is then shifted so that the origin is the center of mass, and a change of basis is applied so that $X^TX = I$. However, if $h_\Gamma (X) > h_\Gamma(X_C)$, then clearly $X_C$ is a better initial approximation, and we still replace $X$ by $X_C$. We then perform some form of smoothing to our embedding $X$, resulting in a new embedding $\widehat X$. If $\widehat X$ is non-planar, the algorithm terminates and outputs $X$. If $\widehat X$ is planar, we again apply some procedure to transform $\widehat X$ into a convex embedding, if it is not already convex. Now that we have a convex embedding $\widehat X$, we shift $\widehat X$ and apply a change of basis, so that $\widehat X^T{\bf 1} =0$ and $\widehat X^T \widehat X = I$. If $h_\Gamma (\widehat X) < h_\Gamma (X)$, then we replace $X$ by $\widehat X$ and repeat this smoothing procedure, producing a new $\widehat X$, until the algorithm terminates. If $h_\Gamma (\widehat X) \ge h_\Gamma (X)$, then we terminate the algorithm and output $X$.

It is immediately clear from the statement of the algorithm that the following result holds.

\begin{proposition}
	The embedding $X_{alg}$ of Algorithm \ref{alg1} satisfies $h_\Gamma(X_{alg}) \le c_1 c_2 \min_{ X_\Gamma \in \mathcal{X}} h_\Gamma (X_\Gamma)$.
\end{proposition}

We now discuss some of the finer details of Algorithm \ref{alg1}. Determining whether an embedding is planar can be done in near-linear time using the sweep line algorithm \cite{shamos1976geometric}. If the embedding is planar, testing if it is also convex can be done in linear time. One such procedure consists of shifting the embedding so the origin is the mass center, checking if the angles each vertex makes with the x-axis are are properly ordered, and then verifying that each vertex $x_i$ is not in $\text{conv}(\{o,x_{i-1},x_{i+1}\})$. Also, in practice, it is advisable to replace conditions of the form $h_\Gamma(X) - h_\Gamma (\widehat X) >0$ in Algorithm \ref{alg1} by the condition $h_\Gamma(X) - h_\Gamma (\widehat X) >
\text{tol}$ for some small value of $\text{tol}$, in order to ensure that the algorithm terminates after some finite number of steps.

There are a number of different choices for smoothing procedures and techniques to make a planar embedding convex. For the numerical experiments that follow, we simply consider the smoothing operation $X \leftarrow S_\Gamma^{-1} X$, and make a planar embedding convex by replacing the embedding by its convex hull, and place vertices equally spaced along each line. For example, if $x_1$ and $x_5$ are vertices of the convex hull, but $x_2,x_3,x_4$ are not, then we set $x_2 = 3/4 x_1 + 1/4 x_5$, $x_3 = 1/2 x_1 + 1/2 x_5$, and $x_4 = 1/4 x_1 + 3/4 x_5$. Given the choices of smoothing and making an embedding convex that we have outlined, the version of Algorithm \ref{alg1} that we are testing has complexity near-linear in $n$. The main cost of this procedure is the computations that involve $S_\Gamma$.

All variants of Algorithm \ref{alg1} require the repeated application of $S_\Gamma$ or $S_\Gamma^{-1}$ to a vector in order to compute the minimal eigenvectors of $S_\Gamma$ (possibly also to perform smoothing). The Schur complement $S_\Gamma$ is a dense matrix and requires the inversion of a $n \times n$ matrix, but can be represented as the composition of functions of sparse matrices. In practice, $S_\Gamma$ should never be formed explicitly. Rather, the operation of applying $S_\Gamma$ to a vector $x$ should occur in two steps. First, the sparse Laplacian system $(L_o + D_o) y = A_{o,\Gamma} x$ should be solved for $y$, and then the product $S x$ is given by $S_\Gamma x = (L_{\Gamma} + D_{\Gamma}) x - A_{o,\Gamma}^T y$. Each application of $S_\Gamma$ is therefore an $O(n \log n)$ procedure (using an $O(n \log n)$ Laplacian solver). The application of the inverse $S_\Gamma^{-1}$ defined on the subspace $\{x \, | \, \langle x, {\bf 1} \rangle = 0 \}$ also requires the solution of a Laplacian system. As noted in \cite{wu2017fourier}, the action of $S_\Gamma^{-1}$ on a vector $x \in \{x \, | \, \langle x, {\bf 1} \rangle = 0 \} $ is given by
$$ S_\Gamma^{-1} x = \begin{pmatrix} 0 & I \end{pmatrix} \begin{pmatrix} L_{o} + D_{o} & -A_{o,\Gamma} \\ -A_{o,\Gamma}^T &
L_{\Gamma} + D_\Gamma \end{pmatrix}^{-1} \begin{pmatrix} 0 \\ x \end{pmatrix}, $$
as verified by the computation
\begin{eqnarray*}
S_\Gamma \left[ S_\Gamma^{-1} x \right] &=& S_\Gamma \begin{pmatrix} 0 & I \end{pmatrix} \left[ \begin{pmatrix} I & 0 \\ -A_{o,\Gamma}^T(L_{o} + D_{o})^{-1} & I \end{pmatrix} \begin{pmatrix} L_{o} + D_{o} & -A_{o,\Gamma} \\ 0 &
S_\Gamma \end{pmatrix} \right]^{-1} \begin{pmatrix} 0 \\ x \end{pmatrix} \\
&=& S_\Gamma \begin{pmatrix} 0 & I \end{pmatrix} \begin{pmatrix} L_{o} + D_{o} & -A_{o,\Gamma} \\ 0 &
S_\Gamma \end{pmatrix}^{-1}  \begin{pmatrix} I & 0 \\ A_{o,\Gamma}^T(L_{o} + D_{o})^{-1} & I \end{pmatrix}  \begin{pmatrix} 0 \\ x \end{pmatrix} \\
&=& S_\Gamma \begin{pmatrix} 0 & I \end{pmatrix} \begin{pmatrix} L_{o} + D_{o} & -A_{o,\Gamma} \\ 0 &
S_\Gamma \end{pmatrix}^{-1}  \begin{pmatrix} 0 \\ x \end{pmatrix} \,= \, x.
\end{eqnarray*}
Given that the application of $S_\Gamma^{-1}$ has the same complexity as an application $S_\Gamma$, the inverse power method is naturally preferred over the shifted power method for smoothing.

\subsection{Numerical Results}

\begin{table}[!t] \label{tb:numerics}
\begin{center}
\begin{tabular}{ | c  c | c c c c c | c c c c c |}

 \hline & & \multicolumn{5}{|c|}{Unit Circle} & \multicolumn{5}{c|}{$3\times1$ Rectangle} \\ 
\multicolumn{2}{|c|}{$n = $}  & 1250 & 2500 & 5000 & 10000 & 20000 & 1250 & 2500 & 5000 & 10000 & 20000 \\ \hline

 \%   & $X_s$ & 100 & 100 & 100 & 100 & 100 & 100 & 100 & 98 & 98 & 97 \\
planar & $X_l$ & 100 & 100 & 100 & 100 & 100 & 67 & 67 & 65 & 71 & 67 \\ \hline

crossings & $X_s$ & n/a	& n/a & n/a &	n/a & n/a & n/a & n/a & 0.042 &	0.062 & 0.063 \\
per edge & $X_l$ & n/a & n/a &	n/a	& n/a &	n/a & 0.143 & 0.119 & 0.129 & 0.132 & 0.129 \\ \hline

\# not & $X_s$ & 0.403 & 0.478 & 0.533 &	0.592 &	0.645 &	0.589 & 0.636 & 0.689 & 0.743	& 0.784 \\
convex & $X_l$ & 0.001 & 0 & 0 & 0 & 0	& 0.397	& 0.418 & 0.428 & 0.443 & 0.448 \\ \hline

 & $X_l$ & 1.026 & 1.024 & 1.02 &	1.017 & 1.015 &	1.938 &	2.143 &	2.291 &	2.555 & 2.861 \\
energy & $X_{sc}$ & 1.004 &	1.004 &	1.004 & 1.004	& 1.003	& 1.127 & 1.164 & 1.208 & 1.285 &	1.356 \\
 ratio & $X_{alg}$ & 1.004 & 1.004 & 1.004 &	1.004 &	1.003 &	1.124 &	1.158 & 1.204 &	1.278 &	1.339 \\
  & $X_{lc}$ & 1.026 & 1.0238 & 1.02 &	1.017 &	1.015 &	1.936 &	2.163 & 2.301 &	2.553 &	2.861 \\
& $ X_C $ & 1.023 & 1.023 & 1.02 &	1.017 &	1.016 &	1.374 & 1.458 &	1.529 &	1.676 &	1.772 \\ \hline
\multicolumn{1}{c}{} \vspace{1 mm}
\end{tabular}
\end{center}

\caption{Numerical results for experiments on Delaunay triangulations of $n$ points randomly generated in a disk or rectangle. One hundred experiments were performed for each convex body and choice of $n$. The row ``\% planar" gives the percent of the samples for which the boundary embedding was planar. The row ``crossings per edge" reports the average number of edge crossings per edge, where the average is taken over all non-planar embeddings. In some cases all one hundred experiments result in planar embeddings, in which case this entry does not contain a value. The row ``\# not convex" reports the average fraction of vertices which are not vertices of the resulting convex hull. This average is taken over all planar embeddings. The row ``energy ratio" reports the average ratio between the value of the objective function $h_\Gamma(\boldsymbol{\cdot})$ for the embedding under consideration and $h_\Gamma(X_s)$. This, again, is an average over all planar embeddings.}
\end{table}

We perform a number of simple experiments, which illustrate the benefits of using the Schur complement to produce an embedding. In particular, we consider the same two types of triangulations as in Figure \ref{fig1}, random triangulations of the unit disk and the $3$-by-$1$ rectangle. For each of these two convex bodies, we sample $n$ points uniformly at random and compute a Delaunay triangulation. For each triangulation, we compute the minimal two non-trivial eigenvectors of the graph Laplacian $L_G$, and the minimal two non-trivial eigenvectors of the Schur complement $S_\Gamma$ of the Laplacian $L_G$ with respect to the interior vertices $V \backslash \Gamma$. The properly normalized and shifted versions of the Laplacian and Schur complement embeddings are denoted by $X_l$ and $X_s$, respectively. We then check whether each of these embeddings of the boundary is planar. If the embedding is not planar, we note how many edge crossings the embedding has. If the embedding is planar, we also determine if it is convex, and compute the number of boundary vertices which are not vertices of the convex hull. If the embedding is planar, but not convex, then we simply replace it by the embedding corresponding to the convex hull of the original layout (as mentioned in Subsection 4.2). This convex-adjusted layout of the Laplacian and Schur complement embeddding (shifted and properly scaled) is denoted by $X_{lc}$ and $X_{sc}$, respectively. The embedding defined by minimal two non-trivial eigenvectors of the boundary Laplacian $L_\Gamma$, denoted by $X_C$, is the typical circular embedding of a cycle (defined in Subsection 4.1). Of course the value $h_\Gamma(X_s)$ is a lower bound for the minimum of (\ref{p1}), and this estimate is exact if $X_s$ is a planar and convex embedding. The embedding resulting from Algorithm \ref{alg1} is denoted by $X_{alg}$. For each triangulation, we compute the ratio of $h_\Gamma(X_s)$ to $h_\Gamma(X_l)$, $h_\Gamma(X_{sc})$, $h_\Gamma(X_{alg})$, $h_\Gamma(X_{lc})$, and $h_\Gamma(X_C)$, conditional on each of these layouts being planar. We perform this procedure one hundred times each for both convex bodies and a range of values of $n$. We report the results in Table \ref{tb:numerics}.

\begin{figure}[t] 
	\begin{center}
		\subfloat[Laplacian Embedding]{\includegraphics*[width=2.5in,height = 1.25 in ]{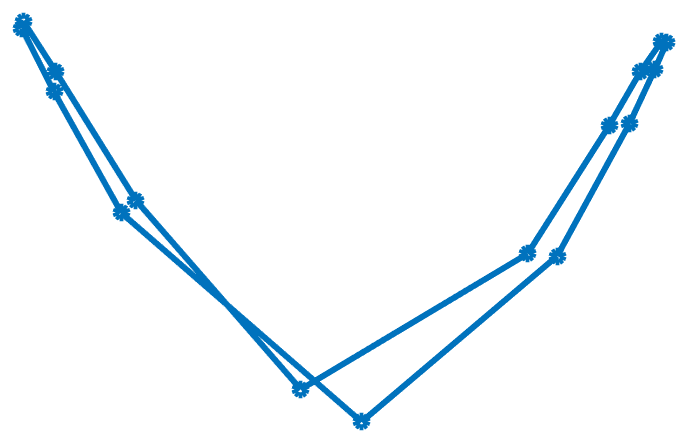}}
		\qquad \qquad
		\subfloat[Schur Complement Embedding]{\includegraphics*[width=2.5 in, height = 1.25 in]{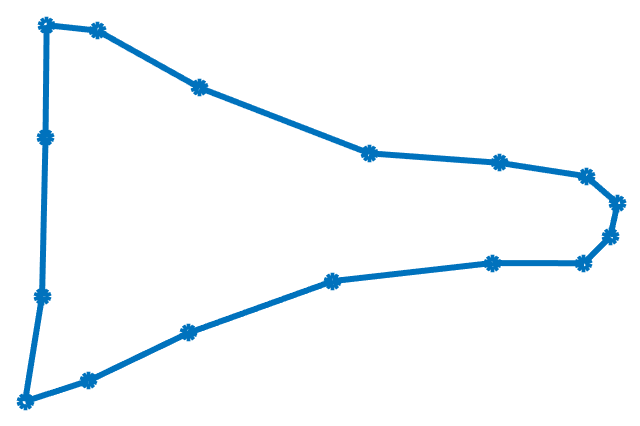}}
		\caption{An example of the Laplacian embedding $X_l$ vs the (unsmoothed) Schur complement embedding $X_s$ of the boundary of the Delaunay triangulation of 1250 points randomly generated in a $3\times1$ rectangle. The Laplacian embedding is non-planar, and far from convex. The Schur complement embedding is planar and almost a convex embedding.}
		\label{fig:rect}
	\end{center}
\end{figure}

These numerical results illustrate a number of phenomena. For instance, when considering the disk both the Laplacian embedding $X_{l}$ and Schur complement $X_{s}$ are always planar, usually close to convex, and their convex versions ($X_{lc}$ and $X_{sc}$) both perform reasonably well compared to the lower bound $h_\Gamma(X_s)$ for Problem (\ref{p1}). The embedding $X_{alg}$ from Algorithm \ref{alg1} produced small improvements over the results of the Schur complement, but this improvement was negligible when average ratio was rounded to the thousands place. As expected, the $L_\Gamma$-based embedding $X_C$ performs well in this instance, as the original embedding of the boundary in the triangulation is already a circle.  Most likely, any graph which possesses a very high level of macroscopic symmetry shares similar characteristics. However, when we consider the rectangle, the convex version of the Schur complement embedding has a significantly better performance than the Laplacian-based embedding. In fact, for a large percentage of the simulations the Laplacian based-embedding $X_l$ was non-planar, and possessed a relatively large number of average crossings per edge. We give a visual representation of the typical difference in the Laplacian vs Schur complement embeddings of the boundary in Figure \ref{fig:rect}. In addition, in this instance, the smoothing procedure of Algorithm \ref{alg1} leads to small, but noticeable improvements. Of course, the generic embedding $X_C$ performs poorly in this case, as the embedding does not take into account any of the dynamics of the interior. 

The Schur complement embedding clearly outperforms the Laplacian embedding, especially for triangulations of the rectangle. From this, we can safely conclude that Laplacian embedding is not a reliable method to embed graphs, and note that, while spectral equivalence does not imply that the minimal two non-trivial eigenvectors produce a planar, near-convex embedding, practice illustrates that for well behaved graphs with some level of structure, this is a likely result.

\section*{Acknowledgements}
The work of L. Zikatanov was supported in part by NSF grants DMS-1720114 and DMS-1819157. The work of J. Urschel was supported in part by ONR Research Contract N00014-17-1-2177. An anonymous reviewer made a number of useful comments which improved the narrative of the paper. The authors are grateful
to Louisa Thomas for greatly improving the style of presentation.

\bibliographystyle{plain}
\bibliography{graphs} 

\appendix

\section{Technical Trace Theorem Proofs} 

\begin{lemma}\label{app:lm1} Let $G =G_{k,\ell}$, $ 4 \ell <k < 2 c \ell$, $c \in \mathbb{N}$, with boundary $\Gamma = \{(i,1)\}_{i=1}^k$. For any $\varphi \in \R{k}$, the vector $u$ defined by (\ref{eqn:ext}) satisfies
	$$ |u|_G \le \sqrt{2c + \frac{233}{9}} \, |\varphi|_\Gamma.$$
\end{lemma}

\begin{proof}
	We can decompose $|u|^2_G$ into two parts, namely, 
	$$   |u|^2_G =
	\sum_{i=1}^{k}\sum_{j=1}^{\ell}
	(u(i+1,j)-u(i,j))^2 +
	\sum_{i=1}^{k}\sum_{j=1}^{\ell-1}(u(i,j+1)-u(i,j))^2.$$
	We bound each sum separately, beginning with the first. We have 
	\begin{eqnarray*}
		u(i+1,j)-u(i,j) & = &
		\left(1-\frac{j-1}{\ell-1}\right)(a(i+1,j)-a(i,j)) \\
		& = & 
		\left(1-\frac{j-1}{\ell-1}\right)\frac{\varphi(i+j)-\varphi(i+1-j)}{2j-1}.
	\end{eqnarray*}
	Squaring both sides and noting that $4 \ell < k$, we have
	\begin{eqnarray*}
		\sum_{i=1}^{k} \sum_{j=1}^{\ell} (u(i+1,j)-u(i,j))^2 &\le& \sum_{i=1}^{k} \sum_{j=1}^{\ell} \left[\frac{\varphi(i+j)-\varphi(i+1-j)}{2j-1} \right]^2 \\
		&\le& \sum_{p=1}^k \sum_{h=1}^{2 \ell-1} \left[\frac{\varphi(p+h)-\varphi(p)}{h} \right]^2 \, \le \,  | \varphi |^2_{\Gamma}.
	\end{eqnarray*}
	We now consider the second sum. Each term can be decomposed as
	$$   u(i,j+1)-u(i,j)  =  \frac{a-a(i,j)}{\ell-1}
	+ \left(1-\frac{j}{\ell-1}\right)[a(i,j+1) - a(i,j)],$$
	which leads to the upper bound
	$$  \sum_{i=1}^{k} \sum_{j=1}^{\ell-1}
	(u(i,j+1)-u(i,j))^2 \le
	2 \sum_{i=1}^{k} \sum_{j=1}^{\ell-1}
	\left[\frac{a-a(i,j)}{\ell-1}\right]^2 +
	2 \sum_{i=1}^{k} \sum_{j=1}^{\ell-1} (a(i,j+1)-a(i,j))
	^2.$$
	We estimate these two terms in the previous equation separately, beginning with the first. The difference $a-a(i,j)$ can be written as
	\begin{eqnarray*}
		a-a(i,j) &=&
		\frac{1}{k} \sum_{p=1}^{k} \varphi(p) - 
		\frac{1}{2j-1} \sum_{h=1-j}^{j-1} \varphi(i+h) \\
		& = & 
		\frac{1}{k(2j-1)} \sum_{p=1}^{k} \sum_{h=1-j}^{j-1} \varphi(p) - \varphi(i+h).
	\end{eqnarray*}
	Squaring both sides,
	\begin{eqnarray*}
		(a-a(i,j))^2 &=&
		\frac{1}{k^2(2j-1)^2}\left(\sum_{p=1}^{k} \sum_{h=1-j}^{j-1} \varphi(p) - \varphi(i+h) \right)^2\\
		& \le & 
		\frac{1}{k(2j-1)}\sum_{p=1}^{k} \sum_{h=1-j}^{j-1}(\varphi(p) - \varphi(i+h))^2.
	\end{eqnarray*}
	Summing over all $i$ and $j$ gives
	\begin{eqnarray*}
		\sum_{i=1}^{k}
		\sum_{j=1}^{\ell-1}\left[\frac{(a-a(i,j))}{\ell-1}\right]^2  & \le & 
		\frac{1}{(\ell-1)^2} \sum_{i=1}^{k}
		\sum_{j=1}^{\ell-1} \frac{1}{k(2j-1)} \sum_{p=1}^k \sum_{h=1-j}^{j-1}(\varphi(p)-\varphi(i+h))^2\\
		& = &
		\frac{k}{4 (\ell-1)^2}\sum_{j=1}^{\ell-1} \frac{1}{2j-1}\sum_{h=1-j}^{j-1} \sum_{i,p=1}^{k} \frac{(\varphi(p)-\varphi(i+h))^2}{k^2/4}\\
		&\le & \frac{k}{4(\ell-1)} |\varphi|_{\Gamma}^2 \, \le \, c |\varphi|_{\Gamma}^2.
	\end{eqnarray*}
	This completes the analysis of the first term. For the second term, we have
	\begin{eqnarray*}
		a(i,j+1) - a(i,j)& = & 
		\frac{1}{2j+1}\left[\varphi(i+j)+\varphi(i-j)-\frac{2}{2j-1}\sum_{h=1-j}^{j-1} \varphi(i+h)\right].
	\end{eqnarray*}
	Next, we note that  
	\begin{eqnarray*}
		\left| \varphi(i+j)-\frac{1}{2j-1}\varphi(i)-
		\frac{2}{2j-1}\sum_{h=1}^{j-1} \varphi(i+h) \right|
		&=& 
		\left| \frac{\varphi(i+j)-\varphi(i)}{2j-1}+
		2\sum_{h=1}^{j-1} \frac{\varphi(i+j)-\varphi(i+h)}{2j-1} \right| \\
		& \le  & 
		2\sum_{h=0}^{j-1} \frac{\left| \varphi(i+j)-\varphi(i+h) \right|}{2j-1},
	\end{eqnarray*}
	and, similarly,
	\begin{eqnarray*}
		\left| \varphi(i-j)-\frac{1}{2j-1}\varphi(i)-
		\frac{2}{2j-1}\sum_{h=1}^{j-1} \varphi(i-h) \right|
		&=& 
		\left| \frac{\varphi(i-j)-\varphi(i)}{2j-1}+
		2\sum_{h=1}^{j-1} \frac{\varphi(i-j)-\varphi(i-h)}{2j-1} \right| \\
		& \le & 
		2\sum_{h=0}^{j-1} \frac{\left| \varphi(i-j)-\varphi(i-h) \right|}{2j-1}.
	\end{eqnarray*}
	Hence, 
	$$  \sum_{j=1}^{l-1}(a(i,j+1) - a(i,j))^2 \le \sum_{j=1}^{l-1}
	\frac{8}{(2j+1)^2}\left[\left(
	\sum_{h=0}^{j-1} \frac{\left| \varphi(i+j)-\varphi(i+h) \right|}{2j-1} \right)^2  +
	\left( \sum_{h=0}^{j-1} \frac{\left| \varphi(i-j)-\varphi(i-h) \right|}{2j-1} \right)^2\right] .$$
	Once we sum over all $i$, the sum of the first and second term are identical, and therefore
	$$\sum_{i=1}^k  \sum_{j=1}^{l-1}(a(i,j+1) - a(i,j))^2 \le 16 \sum_{i=1}^k \sum_{j=1}^{l-1}
	\left(
	\sum_{h=0}^{j-1} \frac{\left| \varphi(i+j)-\varphi(i+h) \right|}{(2j-1)(2j+1)} \right)^2.$$
	We have
	\begin{eqnarray*} \sum_{h=0}^{j-1} \frac{  \left| \varphi(i+j)-\varphi(i+h) \right|}{(2j-1)(2j+1)}  &\le &  \frac{1}{3j} \sum_{p=i}^{i+j-1} \frac{
			\left| \varphi(i+j)-\varphi(p) \right|}{j}  \\
		&\le& \frac{1}{3j} \sum_{p=i}^{i+j-1} \frac{\left| \varphi(i+j)-\varphi(p) \right|}{i+j-p},
	\end{eqnarray*}
	which implies that
	\begin{eqnarray*}
		16 \sum_{i=1}^k \sum_{j=1}^{l-1}
		\left(
		\sum_{h=0}^{j-1} \frac{\left| \varphi(i+j)-\varphi(i+h) \right|}{(2j-1)(2j+1)} \right)^2 & \le & \frac{16}{9} \sum_{i=1}^k \sum_{j=1}^{l-1} \left(\frac{1}{j} \sum_{p=i}^{i+j-1} \frac{\left| \varphi(i+j)-\varphi(p) \right|}{i+j-p} \right)^2 \\
		& \le & \frac{16}{9} \sum_{q=1}^{k+\ell-1} \sum_{m=1}^{q-1} \left(\frac{1}{q-m} \sum_{p=m}^{q-1} \frac{\left|\varphi(q)-\varphi(p)\right|}{q-p} \right)^2.
	\end{eqnarray*}
	Letting $r = q - m$, $s = q- p$, and using Hardy's inequality~\cite[Theorem~326]{HardyLittlewoodPolya}, we obtain
	\begin{eqnarray*}
		\frac{16}{9} \sum_{q=1}^{k+\ell-1} \sum_{m=1}^{q-1} \left(\frac{1}{q-m} \sum_{p=m}^{q-1} \frac{\left|\varphi(q)-\varphi(p)\right|}{q-p} \right)^2 &=&
		\frac{16}{9} \sum_{q=1}^{k+\ell-1} \sum_{r=1}^{q-1} \left(\frac{1}{r} \sum_{s=1}^{r} \frac{\left|\varphi(q)-\varphi(q-s)\right|}{s} \right)^2 \\
		&\le& \frac{64}{9} \sum_{q=1}^{k+\ell-1} \sum_{r=1}^{q-1} \left[ \frac{\varphi(q)-\varphi(q-r)}{r}\right]^2 \\
		&=& \frac{32}{9} \sum_{\substack{q_1,q_2=1 \\ q_1 \ne q_2}}^{k+\ell-1} \left[ \frac{\varphi(q_1)-\varphi(q_2)}{q_1-q_2}\right]^2 \\
		&\le& \frac{32}{9} \sum_{\substack{q_1,q_2=1 \\ q_1 \ne q_2}}^{k+\ell-1} \left[ \frac{\varphi(q_1)-\varphi(q_2)}{d_G\left((q_1,1),(q_2,1)\right)}\right]^2,
	\end{eqnarray*}
	where, if $q >k$, we associate $(q,1)$ with $(q^*,1)$, where $q^* = q \mod k$ and $1\le q^*\le k$. The previous sum consists of some amount of over-counting, with some terms $(\varphi(q_1)-\varphi(q_2))^2$ appearing eight times. However, the chosen indexing of the cycle $C_k$ is arbitrary. Therefore, we can average over all $k$ different choices of ordering that preserve direction. In particular,
	$$\sum_{i=1}^k  \sum_{j=1}^{l-1}(a(i,j+1) - a(i,j))^2  \le  \frac{32}{9k} \sum_{t = 0}^{k-1} \sum_{\substack{q_1,q_2=1 \\ q_1 \ne q_2}}^{k+\ell-1} \left[ \frac{\varphi(q_1+t)-\varphi(q_2+t)}{d_G\left((q_1,1),(q_2,1)\right)}\right]^2.$$
	For each choice of $t$, there are $\ell-1$ indices which are over-counted by both summations. Let us consider a specific term corresponding to the indices $q_1$ and $q_2$. If neither of these are over-counted indices, the term will appear twice. If exactly one is an over-counted index, the term will appear four times. Finally, if both are over-counted indices, the term will appear eight times. Summing over all choices of $t$ any term appears at most $2(k-\ell) + 8\ell$ times, which leads to the upper bound
	$$\sum_{i=1}^k  \sum_{j=1}^{l-1}(a(i,j+1) - a(i,j))^2  \le \frac{32}{9} \frac{2(k-\ell) + 8\ell}{k} |\varphi|^2_\Gamma < \frac{112}{9} |\varphi|^2_\Gamma.$$
	
	Combining all our estimates, we obtain the desired result 
	$$|u|_G \le \sqrt{2c + \frac{233}{9}} \, |\varphi|_\Gamma. $$
\end{proof}

\begin{theorem}\label{app:thm1}
	If $H$, $G_{k,\ell}\subset H \subset G^*_{k,\ell}$, $4\ell<k<2c\ell$, $c \in \mathbb{N}$, is an $M$-aggregation of $(G,\Gamma) \in \mathcal{G}_n$, then for any $\varphi \in \R{n_\Gamma}$,
	$$ \frac{1}{ 6 M  \sqrt{M+3} \max\{ \sqrt{3c},2 \pi\}} \, |\varphi|_\Gamma \le \min_{u|_{\Gamma} = \varphi} |u|_G  \le 28 M^2 \sqrt{ 3 c + 20}  \, |\varphi|_\Gamma.$$
\end{theorem}

\begin{proof}
	We first prove that there is an extension $u$ of $\varphi$ which satisfies $|u|_G \le c_1 |\varphi|_{\Gamma}$ for some $c_1$. To do so, we define auxiliary functions $\widehat u$ and $\widehat \varphi$ on $(G^*_{2k,\ell},\Gamma_{2k,\ell})$. Let
	$$\widehat \varphi(p) = \begin{cases} \max_{q \in \Gamma \cap a_{(p+1)/2,1}} \varphi(q) \quad \text{if $p$ is odd},\\
	\; \; \min_{q \in \Gamma \cap a_{p/2,1}} \; \; \; \varphi(q) \quad \text{if $p$ is even},
	\end{cases}$$
	and $\widehat u$ be extension (\ref{eqn:ext}) of $\widehat \varphi$. The idea is to upper bound the semi-norm for $u$ by $\widehat u$, for $\widehat u$ by $\widehat \varphi$ (using Corollary \ref{thm:disctrace2}), and for $\widehat \varphi$ by $\varphi$. On each aggregate $a_{i,j}$, let $u$ take values between $\widehat u(2i-1,j)$ and $\widehat u(2i,j)$, and let $u$ equal $a$ on $a_*$. We can decompose $|u|^2_G$ into
	\begin{eqnarray*}
		|u|^2_G & = & \sum_{i=1}^k \sum_{j=1}^\ell \sum_{\substack{p,q \in a_{i,j},\\ p \sim q}} (u(p)-u(q))^2 + \sum_{i=1}^k \sum_{j=1}^\ell \sum_{\substack{p\in a_{i,j},\\ q\in a_{i+1,j},\\ p \sim q}} (u(p)-u(q))^2\\
		&& + \sum_{i=1}^k \sum_{j=1}^{\ell -1} \sum_{\substack{p\in a_{i,j},\\ q\in a_{i-1,j+1},\\ p \sim q}} (u(p)-u(q))^2 + \sum_{i=1}^k \sum_{j=1}^{\ell -1} \sum_{\substack{p\in a_{i,j},\\ q\in a_{i+1,j+1},\\ p \sim q}} (u(p)-u(q))^2 \\
		&& + \sum_{i=1}^k \sum_{j=1}^{\ell -1} \sum_{\substack{p\in a_{i,j},\\ q\in a_{i,j+1},\\ p \sim q}} (u(p)-u(q))^2,
	\end{eqnarray*}
	and bound each term of $|u|^2_G$ separately, beginning with the first. The maximum energy semi-norm of an $m$ vertex graph that takes values in the range $[a,b]$ is bounded above by $(m/2)^2(b-a)^2$. Therefore,
	$$\sum_{\substack{p,q \in a_{i,j}, \\p \sim q}} (u(p)-u(q))^2 \le \frac{M^2}{4} \left( \widehat u(2i-1,j)- \widehat u (2i,j) \right)^2 .$$
	For the second term,
	\begin{eqnarray*}
		\sum_{\substack{p\in a_{i,j},\\ q\in a_{i+1,j}, \\ p \sim q}} (u(p)-u(q))^2  &\le& M^2 \max_{\substack{i_1 \in \{2i-1,2i \},\\ i_2 \in \{2i+1,2i+2 \} } } (\widehat u(i_1,j)- \widehat u(i_2,j))^2 \\ & \le& 3 M^2 \big[ (\widehat u (2i-1,j) - \widehat u(2i,j))^2  + (\widehat u (2i,j) - \widehat u(2i+1,j))^2 \\ &&   + (\widehat u (2i+1,j) - \widehat u(2i+2,j))^2 \big].
	\end{eqnarray*}
	The exact same type of bound holds for the third and fourth terms. For the fifth term,
	$$ \sum_{\substack{p\in a_{i,j},\\ q\in a_{i,j+1},\\ p \sim q}} (u(p)-u(q))^2  \le M^2 \max_{\substack{i_1 \in  \{2i-1,2i \}, \\ i_2 \in  \{2i-1,2i \}}} (\widehat u(i_1,j)- \widehat u(i_2,j+1))^2 ,$$
	and, unlike terms two, three, and four, this maximum appears in $|\widehat u|^2_{G^*_{2k,\ell}}$. Combining these three bounds, we obtain
	$$ |u|_G \le \frac{ \sqrt{73} M }{2} \, |\widehat u|_{G^*_{2k,\ell}}.$$
	Next, we lower bound $|\varphi|_{\Gamma}$ by a constant times $| \widehat \varphi|_{\Gamma_{2k,\ell}}$. By definition, in $\Gamma \cap a_{i,1}$ there is a vertex which takes value  $\widehat \varphi(2i-1)$ and a vertex which takes value $\widehat \varphi(2i)$. This implies that every term in $| \widehat \varphi|_{\Gamma_{2k,\ell}}$ is a term in $|\varphi|_{\Gamma}$, with possibly different denominator.  Distances between vertices on $\Gamma$ can be decreased by at most a factor of $2 M$ on $\Gamma_{2k,\ell}$. In addition, it may be the case that an aggregate contains only one vertex of $\Gamma$, which results in $\widehat \varphi(2i-1) = \widehat \varphi(2i)$. Therefore, a given term in $|\varphi|^2_\Gamma$ could appear four times in $| \widehat \varphi|^2_{\Gamma_{2k,\ell}}$. Combining these two facts, we immediately obtain the bound
	$$ | \widehat \varphi|_{\Gamma_{2k,\ell}} \le 4 M |\varphi|_\Gamma,$$
	which gives the estimate
	$$|u|_G  \le \frac{\sqrt{73} M}{2} |\widehat u|_{G^*_{2k,\ell}} \le  \frac{\sqrt{73} M}{2} \sqrt{8c + \frac{475}{9}} | \widehat \varphi|_{\Gamma_{2k,\ell}} \le  28 M^2 \sqrt{ 3 c + 20} \, |\varphi|_\Gamma,$$
	where we have slightly increased the constants in the last inequality, for the sake of presentation. This completes the first half of the proof.
	
	All that remains is to show that for any $u$, $| \varphi|_{\Gamma} \le c_2 |u|_G$ for some $c_2$. To do so, we define auxiliary functions $\widetilde u$ and $\widetilde \varphi$ on $(G_{2k,2\ell},\Gamma_{2k,2\ell})$. Let
	$$ \widetilde u(i,j) = \begin{cases} \max_{p \in a_{\lceil i/2 \rceil,\lceil j/2 \rceil}} u(p) \quad \text{ if } i = j \mod 2, \\
	\min_{p \in a_{\lceil i/2 \rceil,\lceil j/2 \rceil}} u(p) \quad \text{ if } i \ne j \mod 2. \end{cases}$$
	Here, the idea is to lower bound the semi-norm for $u$ by $\widetilde u$, for $\widetilde u$ by $\widetilde \varphi$ (using Corollary \ref{thm:disctrace2}), and for $\widetilde \varphi$ by $\varphi$. We can decompose $|\widetilde u|^2_{G_{2k,2\ell}}$ into
	\begin{eqnarray*}
		|\widetilde u|^2_{G_{2k,2\ell}} &=& 4 \sum_{i=1}^k \sum_{j=1}^\ell \left(\widetilde u(2i-1,2j-1) - \widetilde u(2i,2j-1)\right)^2 \\
		&& + \sum_{i=1}^k \sum_{j=1}^\ell ( \widetilde u(2i,2j-1) - \widetilde u(2i+1,2j-1))^2 + (\widetilde u(2i,2j) - \widetilde u(2i+1,2j))^2 \\
		&& + \sum_{i=1}^k \sum_{j=1}^{\ell-1} (\widetilde u(2i-1,2j) - \widetilde u(2i-1,2j+1))^2 + (\widetilde u(2i,2j) - \widetilde u(2i,2j+1))^2,
	\end{eqnarray*}
	and bound each term separately, beginning with the first. The minimum squared energy semi-norm of an $m$ vertex graph that takes value $a$ at some vertex and value $b$ at some vertex is bounded below by $(b-a)^2/m$. Therefore,
	$$  (\widetilde u(2i-1,2j-1) - \widetilde u(2i,2j-1))^2 \le M \sum_{\substack{p,q \in a_{i,j},\\ p \sim q}} (u(p)-u(q))^2 .$$
	For the second term, we first note that
	$$ \min_{\substack{i_1 \in \{2i-1,2i\},\\ i_2 \in \{2i+1,2i+2\}}} (\widetilde u(i_1,2j)- \widetilde u(i_2,2j))^2 \le \sum_{\substack{p\in a_{i,j}, \\ q\in a_{i+1,j}, \\ p \sim q}} (u(p)-u(q))^2.$$
	One can quickly verify by application of Cauchy-Schwarz that
	$$ \left( \widetilde u(2i-1,2j) - \widetilde u(2i+2,2j)\right)^2 + \left(\widetilde u(2i,2j) - \widetilde u(2i+1,2j) \right)^2 $$
	is bounded above by 
	$$3(\widetilde u(2i-1,2j) -  \widetilde u(2i,2j))^2  + 3 (\widetilde u(2i+1,2j) -  \widetilde u(2i+2,2j))^2  + 4 \min_{\substack{i_1 \in \{2i-1,2i\},\\ i_2 \in \{2i+1,2i+2\}}} (\widetilde u(i_1,2j)- \widetilde u(i_2,2j))^2.$$
	The technique for the third term is identical to that of the second term. Therefore,
	$$|\widetilde u|_{G_{2k,2\ell}} \le  2 \sqrt{ M+3} \, |u|_G. $$
	Next, we upper bound $|\varphi|_{\Gamma}$ by a constant multiple of $|\widetilde \varphi|_{\Gamma_{2k,2\ell}}$. We can write $|\varphi|^2_{\Gamma}$ as
	$$ | \varphi|^2_{\Gamma} = \sum_{i=1}^k \sum_{p, q \in \Gamma \cap a_{i,1}} \frac{(\varphi(p)-\varphi(q))^2}{d^2_G(p,q)} + \sum_{i_1=1}^{k-1} \sum_{i_2= i_1 + 1}^k \sum_{\substack{p \in \Gamma \cap a_{i_1,1},\\ q \in \Gamma \cap a_{i_2,1}}} 
	\frac{(\varphi(p)-\varphi(q))^2}{d^2_G(p,q)} ,$$
	and bound each term separately. The first term is bounded by
	$$ \sum_{p, q \in \Gamma \cap a_{i,1}} \frac{(\varphi(p)-\varphi(q))^2}{d^2_G(p,q)} \le \frac{ M^2}{4} (\widetilde \varphi(2i-1) - \widetilde \varphi(2i))^2.$$
	For the second term, we first note that $d_G(p,q) \ge 3 d_{\Gamma_{2k,2\ell}}\left((m_1,1),(m_2,1)\right)$ for  $p \in \Gamma \cap a_{i_1,1}$, $q \in \Gamma \cap a_{i_2,1}$, $m_1 \in \{2i_1-1,2i_1\}$, $m_2 \in \{2i_2-1,2i_2\}$, which allows us to bound the second term by
	$$\sum_{\substack{p \in \Gamma \cap a_{i_1,1}, \\ q \in \Gamma \cap a_{i_2,1}}} 
	\frac{(\varphi(p)-\varphi(q))^2}{d^2_G(p,q)}  \le 9 M^2 \max_{\substack{m_1 \in \{2i_1-1,2i_1\},\\ m_2 \in \{2i_2-1,2i_2\}}} \frac{(\widetilde \varphi(m_1) - \widetilde \varphi(m_2))^2}{d^2_{\Gamma_{2k,2\ell}}(m_1,m_2)}.$$
	This immediately implies that
	$$ | \varphi |_{\Gamma} \le 3 M | \widetilde \varphi|_{\Gamma_{2k,2\ell}},$$
	and, therefore,
	$$ | \varphi |_{\Gamma} \le 3 M | \widetilde \varphi|_{\Gamma_{2k,2\ell}} \le 3 M \max\{ \sqrt{3c},2 \pi\} |\widetilde u |_{G_{2k,2\ell}} \le 6 M  \sqrt{M+3} \max\{ \sqrt{3c},2 \pi\} \, |u|_G. $$
	This completes the proof.
\end{proof}

\begin{theorem} \label{app:thm2}
	If there exists a planar spring embedding $X$ of $(G, \Gamma) \in \mathcal{G}_n^{f\le c_1}$ for which
	\begin{enumerate}[(1)]
		\item $K= \text{conv}\left(\{ [X_\Gamma]_{i,\boldsymbol{\cdot}} \}_{i=1}^{n_\Gamma} \right)$ satisfies  
		$$\sup_{u \in K} \inf_{v \in \partial K} \sup_{w \in \partial K} \frac{\|u - v \| }{ \|u - w \|} \ge c_2>0,$$
		\item $X$ satisfies
		$$ \max_{\substack{\{i_1,i_2\} \in E \\ \{j_1,j_2\} \in E}} \frac{ \|X_{i_1,\boldsymbol{\cdot}} - X_{i_2,\boldsymbol{\cdot}} \|}{ \|X_{j_1,\boldsymbol{\cdot}} - X_{j_2,\boldsymbol{\cdot}} \|} \le c_3 \quad \text{and} \quad 
		\min_{\substack{i \in V \\ j_1,j_2 \in N(i)}} \angle \, X_{j_1,\boldsymbol{\cdot}}  \, X_{i,\boldsymbol{\cdot}} \, X_{j_2,\boldsymbol{\cdot}}  \ge c_4>0,$$
	\end{enumerate}
	then there exists an $H$, $G_{k,\ell}\subset H \subset G^*_{k,\ell}$, $
	\ell \le k<2c\ell$, $c \in \mathbb{N}$, such that $H$ is an $M$-aggregation of $(G,\Gamma)$ where $c$ and $M$ are constants that depend on $c_1$, $c_2$, $c_3$, and $c_4$.
\end{theorem}

\begin{proof}

This proof consists of three main parts. First, we will prove some basic properties regarding the embedding $X$. Second, we will partition $K$ into subregions, and prove a number of properties regarding these subregions. Third, we will use these subregions to define a partition of the vertex set of $G$, and prove that this partition is an $M$-aggregation.

The majority of the estimates that follow are not tight, and due to the long nature of this proof, simplicity is always preferred over improved constants. This proof relies on a sufficiently large dimension $n$, so that functions of $c_1$, $c_2$, $c_3$, and $c_4$ are sufficiently small in comparison. If at any point during the course of the proof this assumption does not hold, then we may conclude that at least one of these constants depends on $n$, and may take $M = n$, thus completing the proof.

We begin by proving a number of preliminary estimates, obtained by simple geometry. The conditions of the theorem do not depend on the scale or relative location of the embedding $X$, so, without loss of generality, we may suppose that the choice of $u$ which maximizes $(1)$ is the origin $o$, and that the minimum edge length
	$$ \min_{(i_1,i_2) \in E} \|X_{i_1,\boldsymbol{\cdot}} - X_{i_2,\boldsymbol{\cdot}} \| = 1.$$
	We now state a number of basic facts. \\

{\bf Fact 1:} The maximum edge length is at most $c_3$. \\
\indent {\bf Fact 2:} The diameter of every inner face is at most $\tfrac{1}{2}c_1 c_3$. \\
\indent {\bf Fact 3:} The area of each interior face is at least $a_1 := \tfrac{1}{2} \sin c_4 $. \\
\indent {\bf Fact 4:} The area of each interior face is at most $ a_2: =\tfrac{1}{4} c_1 c_3^2  \cot\tfrac{\pi}{c_1}$. \\
\indent {\bf Fact 5:} $G$ has at least $\tfrac{2}{c_1} n$ faces and at most $2 n$ faces. \\
\indent {\bf Fact 6:} The area of $K$ is at least $\tfrac{2}{c_1} a_1 n$ and at most $2 a_2 n$. \\

Fact 1 follows from condition (2). Fact 2 is based on the upper bounds $c_1$ on edge lengths and $c_3$ on number of edges in an inner face. The lower bound in Fact 3 is the area of a triangle with two sides of length one and internal angle $c_4$, a triangle which is contained, by assumption, in every inner face. The upper bound in Fact 4 is the area of a regular $c_1$-gon with side lengths $c_3$. Fact 5 follows from Euler's formula and three-connectedness. Fact 6 is simply an application of Fact 5 and the upper and lower bounds on the area of an inner face.

Using Fact 6, we can upper bound the distance and lower bound the Hausdorff distance (denoted by $d^H(\boldsymbol{\cdot}, \boldsymbol{\cdot})$) between $o$ and $\partial K$ by
	$$  d(o,\partial K) \le \sqrt{\frac{2 a_2 n }{\pi}}  \quad \text{ and } \quad d^{H} (o, \partial K) \ge \sqrt{\frac{2 a_1 n }{\pi c_1}}.$$
	Combining these inequalities with condition (1), we obtain the estimates
	$$ c_2 \sqrt{\frac{2 a_1 n }{\pi c_1}} \le d(o,\partial K) \le \sqrt{\frac{2 a_2 n }{\pi}} \qquad \text{and} \qquad
	\sqrt{\frac{2 a_1 n }{\pi c_1}} \le d^{H} (o, \partial K) \le \frac{1}{c_2} \sqrt{\frac{2 a_2 n }{\pi}}. $$

    Let us write points $x \in \mathbb{R}^2$ in polar coordinates $x = (r,\theta)$. Define $\partial K_{\theta}$ to be the unique $x \in \partial K$ satisfying $x = (r,\theta)$, and $\partial K_{\theta_1,\theta_2}$ to be the shortest curve between $\partial K_{\theta_1}$ and $\partial K_{\theta_2}$ lying entirely in $\partial K$. The boundary $\partial K$ is contained in the annulus
    $$ \partial K \subset \left\{ x \, \Bigg| \, c_2 \sqrt{\frac{2 a_1 n }{\pi c_1}}  \le \| x \| \le  \frac{1}{c_2} \sqrt{\frac{2 a_2 n }{\pi}} \right\},$$
    and, therefore, by the convexity of $K$, 
    the angle $\angle \, o \, \partial K_{\theta_1} \, \partial K_{\theta_2}$ is bounded away from $0$ and $ \pi$, say 
    $$0< c_5 \le \angle \, o \, \partial K_{\theta_1} \, \partial K_{\theta_2} \le \pi - c_5 < \pi$$
    for all $|\theta_1-\theta_2| \le \pi/4$ and some constant $c_5$ that is independent of $n$. The condition on the distance between $\theta_1$ and $\theta_2$ is arbitrary, but avoids having two points on opposite sides of $o$.
    
    We are now prepared to define a partition of $K$. Let $k,\ell \in \mathbb{N}$ equal
    $$ k = \left\lfloor \frac{c_2}{c_1 c_3} \sqrt{\frac{\pi a_1 n}{2 c_1}} \right\rfloor \qquad \text{ and } \qquad \ell = \left\lfloor \frac{c_2 \sin c_5}{2 c_1 c_3} \sqrt{\frac{a_1 n}{2 \pi c_1}} \right \rfloor$$
    and define
    \begin{align*}
	K_{i,1} &= \text{conv} \left( \partial K_{\frac{2 \pi (i-1)}{k},\frac{2 \pi i}{k} }, \tfrac{2\ell -1}{2 \ell} \partial K_{\frac{2 \pi (i-1)}{k}}, \tfrac{2\ell -1}{2 \ell}  \partial K_{\frac{2 \pi i}{k}} \right), \qquad i = 1,...,k \\
	K_{i,j} &= \text{conv} \left(\tfrac{2\ell-j+1}{2 \ell} \partial K_{\frac{2 \pi (i-1)}{k}},\tfrac{2\ell-j+1}{2 \ell}\partial K_{\frac{2 \pi i}{k}},\tfrac{2\ell-j}{2 \ell}\partial K_{\frac{2 \pi (i-1)}{k}},\tfrac{2\ell-j}{2 \ell}\partial K_{\frac{2 \pi i}{k}}\right), \quad i = 1,...,k, \; j = 2,...,\ell, \\
	K_a &= \text{conv}\left( \left\{ \tfrac{1}{2} \partial K_{\frac{2 \pi i}{k}} \right\}_{i =1}^k \right).
	\end{align*}
	Suppose that $k, \ell \ge 3$ (if $k$ or $\ell$ is less than three, then as previously mentioned, $c_1,c_2,c_3,c_4$ depend on $n$, and we are done). The triangles $$\triangle \, o \, \left[ \tfrac{2\ell-j+1}{2 \ell} \partial K_{\frac{2 \pi (i-1)}{k}} \right] \, \left[\tfrac{2\ell-j+1}{2 \ell}\partial K_{\frac{2 \pi i}{k}}\right] \qquad \text{ and } \qquad \triangle \, o \, \left[ \tfrac{2\ell-j}{2 \ell} \partial K_{\frac{2 \pi (i-1)}{k}} \right] \, \left[\tfrac{2\ell-j}{2 \ell}\partial K_{\frac{2 \pi i}{k}}\right]$$
	are similar, and therefore the quadrilaterals $K_{i,j}$, $j \ne 1$, are trapezoids with angles in the range $[c_5,\pi-c_5]$. By the lower bound $d(o,\partial K ) \ge  c_2 \sqrt{\frac{2 a_1 n}{\pi c_1}}$ and the formula for a chord of a circle, we can immediately conclude that the length of the sides is at least
    $$d \left(\tfrac{2\ell-j+1}{2 \ell}\partial K_{\frac{2 \pi i}{k}} ,  \tfrac{2\ell-j}{2 \ell}\partial K_{\frac{2 \pi i}{k}}  \right) \ge \frac{c_2 \sqrt{\frac{2 a_1 n}{\pi c_1}}}{2\ell} $$
    and
    $$ d\left( \tfrac{2\ell-j}{2 \ell}\partial K_{\frac{2 \pi (i-1)}{k}} ,  \tfrac{2\ell-j}{2 \ell}\partial K_{\frac{2 \pi i}{k}}  \right) \ge 2 c_2 \sqrt{\frac{2 a_1 n}{\pi c_1}} \sin \frac{\pi}{k} \ge c_2 \sqrt{\frac{2 a_1 n}{\pi c_1}} \frac{\pi}{k} .$$

    Let
    $$ x_{i,j} = \overline{\left[ \tfrac{2\ell-j+1/2}{2 \ell}\partial K_{\frac{2 \pi (i-1)}{k}} \right] \left[ \tfrac{2\ell-j+1/2}{2 \ell}\partial K_{\frac{2 \pi i}{k}} \right]} \; \cap \; \left\{ \left(r,\frac{2 \pi (i-1/2)}{k} \right) \, \Big| \, r >0\right\}, \quad i = 1,...,k, \; j = 1,...,\ell$$
    serve as a ``center" of sorts for each $K_{i,j}$.
    By the same chord argument used above,
    $$ d \left(\text{conv}\left(\left\{x_{i,j} \right\}_{j=1}^\ell \right), \left[ \text{conv} \left(\partial K_{\frac{2\pi i}{k}}, \tfrac{1}{2} \partial K_{\frac{2\pi i}{k}} \right) \cup \text{conv} \left(\partial K_{\frac{2\pi (i-1)}{k}}, \tfrac{1}{2} \partial K_{\frac{2\pi (i-1)}{k}} \right)\right]  \right) \ge c_2 \sqrt{\frac{2 a_1 n}{\pi c_1}} \frac{\pi}{2k}.$$
    In addition, by the formula for the height of a trapezoid and the lower bound on the angles of the trapezoids, we have
    $$ d \left(\overline{\left[ \tfrac{2\ell-j+1/2}{2 \ell}\partial K_{\frac{2 \pi (i-1)}{k}} \right] \left[ \tfrac{2\ell-j+1/2}{2 \ell}\partial K_{\frac{2 \pi i}{k}} \right]}, \overline{\left[ \tfrac{2\ell-j}{2 \ell}\partial K_{\frac{2 \pi (i-1)}{k}} \right] \left[ \tfrac{2\ell-j}{2 \ell}\partial K_{\frac{2 \pi i}{k}} \right]} \right) \ge \frac{c_2 \sqrt{\frac{2 a_1 n}{\pi c_1}}}{4\ell} \sin c_5. $$
    We note that, by the definitions of $k$ and $\ell$, both of the above lower bounds is at least $c_1 c_3$.

    We are now prepared to define our aggregation. We will perform an iterative procedure, in which we grow the aggregates $a_{i,j}$ until we have a partition with our desired properties. First, each $a_{i,j}$ will be a subset of the set of vertices contained in a face that intersects $K_{i,j}$. This condition, combined with Fact 4, proves that $M$ is a constant depending on $c_1,c_2,c_3,c_4$. That $c$ is a constant already follows from the definitions of $k$ and $\ell$. In addition, this condition, paired with the upper bound $\tfrac{1}{2}c_1 c_3$ on the diameter of an inner face and the bounds for the trapezoids, guarantees that non-adjacent aggregates (with respect to $G^*_{k,\ell}$) are not connected. This condition also guarantees that $\Gamma \subset \cup_{i=1}^k a_{i,1}$. All that remains is to show that $G[a_{i,j}]$ is connected and adjacent aggregates (with respect to $G_{k,\ell}$) are connected to each other in our resulting aggregation.
    
    First, let us add to each $a_{i,j}$ all vertices which lie on a face containing the point $x_{i,j}$ (if the embedding of a vertex or edge does not intersect the point $x_{i,j}$, then there is only one such face). By construction, so far $G[a_{i,j}]$ is connected for all $i,j$. To connect adjacent aggregates $a_{i,j}$ and $a_{i,j+1}$ consider all the faces which intersect the line segment $\overline{x_{i,j} x_{i,j+1}}$. Because this set of faces connects $a_{i,j}$ to $a_{i,j+1}$, there exists a shortest path $P:=p_1 \, ... \, p_t$, $p_1 \in a_{i,j}$, $p_t \in a_{i,j+1}$, between $a_{i,j}$ and $a_{i,j+1}$ which only uses vertices in the aforementioned faces. Let $s$ be the smallest index such that $X_{p_s,\boldsymbol{\cdot}} \in K_{i,j+1}$. Add to $a_{i,j}$ the vertices $p_1,...,p_{s-1}$ and to $a_{i,j+1}$ the vertices $p_s,...,p_t$, for every $i=1,...,k$, $j =1,...,\ell-1$. In parallel, also perform a similar procedure for all pairs of aggregates $a_{i,j}$ and $a_{i+1,j}$ by using the union of the line segments $$\overline{x_{i,j} \, \left[ \tfrac{2\ell-j+1/2}{2 \ell}\partial K_{\frac{2 \pi i}{k}} \right]} \qquad \text{and} \qquad  \overline{\left[ \tfrac{2\ell-j+1/2}{2 \ell}\partial K_{\frac{2 \pi i}{k}} \right] x_{i+1,j}}$$
    to connect $x_{i,j}$ and $x_{i+1,j}$. At this point, each $G[a_{i,j}]$ is still connected, and adjacent aggregates are connected. However, not all vertices are in an aggregate yet. To complete the proof, perform a parallel breadth-first search for each $a_{i,j}$ simultaneously and add the vertices of each depth of this search to each $a_{i,j}$ iteratively, while adhering to the condition that $a_{i,j}$ is a subset of the set of vertices contained in a face that intersects $K_{i,j}$. Once the breadth-first search is complete, add all remaining vertices to $a_*$. This completes the proof.
\end{proof}

\end{document}